\documentclass[reqno,11pt]{amsart}
\usepackage{amssymb,amsmath,amscd, amsthm,stmaryrd,verbatim,
mathrsfs,tikz,color , amsfonts, enumerate, esvect
}
\usetikzlibrary{patterns}
\usetikzlibrary{er}
\usepackage{genyoungtabtikz}

\usepackage[pdfborder={0 0 0}]{hyperref}

\numberwithin{equation}{section}
\allowdisplaybreaks

\newcommand{\dom}{\operatorname{dom}}
\newcommand{\caH}{\mathcal{H}}
\newcommand{\sgn}{\operatorname{sgn}}
\newcommand{\Ree}{\mathrm{Re}}

\newcommand{\seq}{ \mathrm{Seq}_n(\mathbb{R})}

\theoremstyle{theorem}

\newtheorem{Thm}{Theorem}[section]
\newtheorem{Prop}[Thm]{Proposition}
\newtheorem{Lem}[Thm]{Lemma}

\theoremstyle{definition}
\newtheorem{Defff}[Thm]{Definition}
\newtheorem{ex}[Thm]{Example}
\newtheorem{Rem}[Thm]{Remark}

\newcommand{\lam}{\lambda}

\newcommand{\Sp}{{Sp(n,\mathbb{R})}}
\newcommand{\mb}{\mathbf}

\newcommand{\al}{\alpha}

\newcommand{\ep}{\varepsilon}

\newcommand{\eq}{\begin{equation}}
\newcommand{\en}{\end{equation}}
\newcommand{\beqna}[1]{\begin{eqnarray}\label{#1}}
\newcommand{\eeqna}{\end{eqnarray}}
\newcommand{\beqn}[1]{\begin{equation}\label{#1}}
\newcommand{\eeqn}{\end{equation}}

\newcommand{\af}{$ \mathbf{a} $}
\newcommand{\mc}[1]{\mathcal{#1}}
\newcommand{\msc}[1]{\mathscr{#1}}
\newcommand{\mf}[1]{\mathfrak{#1}}
\renewcommand{\subset}{\subseteq}
\newcommand{\rar}{\rightarrow}
\newcommand{\bil}[2]{\langle{#1},{#2}^{\vee} \rangle }
\newcommand{\hs}{ \mathfrak{h}^*}

\newcommand{\aff}{ \mathbf{a} }

\newcommand{\cff}{ \mathbf{c} }
\newcommand{\eff}{ \mathbf{e} }
\newcommand{\gkd}{\operatorname{GKdim}}

\newcommand{\hdom}[3]{\draw (0+#1,0-#2) rectangle (2+#1,-1-#2)++(-1,+0.5) node {$ #3$};}
\newcommand{\vdom}[3]{\draw (0+#1,0-#2) rectangle (1+#1,-2-#2)++(-0.5,+1) node {$ #3$};}
\newcommand{\hobox}[3]{\draw (0+#1,0-#2) rectangle (1+#1,-1-#2)++(-0.5,+0.5) node {$ #3$};}

\newcommand{\domscale}{0.51}
\newcommand{\sh}{\operatorname{sh}}
\newcommand{\ev}{^{\mathrm{ev}}}
\newcommand{\od}{^{\mathrm{odd}}}

\begin{document}

\bibliographystyle{alpha}
\title[GK dimensions and associated varieties]{Gelfand-Kirillov dimensions and associated varieties of highest weight modules
}
\author{Zhanqiang Bai}\author{Wei Xiao}\author{Xun Xie}
\address[Bai]{School of Mathematical Sciences, Soochow University, Suzhou 215006, P. R. China}
\email{zqbai@suda.edu.cn}

\address[Xiao]{College of Mathematics and statistics, Shenzhen Key Laboratory of Advanced Machine Learning and Applications, Shenzhen University,
	Shenzhen 518060, Guangdong, P. R. China}
\email{xiaow@szu.edu.cn}

\address[Xie]{School of Mathematics and Statistics, Beijing Institute of Technology, Beijing 100081, China}
\email{xieg7@163.com}
\subjclass[2010]{Primary 22E47; Secondary 17B10, 20C08}
\date{\today}
\keywords{Lusztig's $ \aff $-function,  Gelfand-Kirillov dimension, associated variety, domino insertion, highest weight Harish-Chandra module } 

\maketitle

\begin{abstract}
In this paper, we present a uniform formula of Lusztig's  $ \mathbf{a}$-functions on classical Weyl groups. Then we obtain  an efficient algorithm for the Gelfand-Kirillov dimensions of simple highest weight modules of classical Lie algebras, whose highest weight is not necessarily regular or integral. To deal with type $ D $, we prove an interesting property about domino tableaux associated with Weyl group elements by  introducing an invariant, called the hollow tableau. As an application, the associated varieties of all the simple highest weight Harish-Chandra modules are explicitly determined, including the exceptional cases.
\end{abstract}
\section{Introduction}

\subsection{GK dimensions and associated varieties}

Let $\mathfrak{g}$ be a complex semisimple Lie algebra and $U(\mathfrak{g})$ be its enveloping algebra. The Gelfand-Kirillov (GK) dimension is an essential invariant to measure the size of an infinite-dimensional $U(\mathfrak{g})$-module $M$ \cite{Ge-Ki}. The associated variety of $M$, whose dimension is equal to the GK dimension of $M$, is another invariant to distinguish $U(\mathfrak{g})$-modules \cite{Be}. These two invariants are closely related and play significant roles in representation theory, see for example \cite{J80-1,J80-2,J81-1, Tr-01,LM}. 

The computations of these two invariants are generally complicated. The case of simple highest weight modules for complex simple Lie algebras has attracted considerable amount of attention (e.g., \cite{Jo84,BoB3,Ta,Mel,Mc00,Wi}), as well as the case of Harish-Chandra (HC) modules for real reductive groups (e.g., \cite{Vo78, Vo91, Garfinkle1993, Vo19}). The first main result of this paper is an efficient algorithm for GK dimensions of highest weight modules for classical Lie algebras. A formula of  GK dimensions of regular integral highest weight modules for type $A$ was obtained by Joseph \cite{Jo78}. Based on the theory of Lusztig's $\aff$-function  \cite{lusztig1977symbol, Lu79,Lusztig1982A-class-II, lusztig1984char,lusztig1985cellsI,lusztig1985A_n}, Bai-Xie were able to extend Joseph's result to singular and nonintegral cases for type $A$ \cite{BX}. More efforts are needed for the other classical types. In this paper, formulas of GK dimensions for simple highest weight modules are explicitly presented in the most generality (Theorem \ref{thm2} and \ref{thm:algo}). As a critical step, we give a uniform formula of $\aff$-functions (Theorem \ref{thm1}) on classical Weyl groups.

The second main result of this paper is a full description of the associated varieties of all the highest weight HC modules. Such a result was expected for a long time. In 1992, Joseph \cite{Jo92} found a recursive relation for associated varieties of unitary highest weight modules in a case-by-case fashion. After that, the GK dimensions of holomorphic representations were obtained for Hermitian symmetric groups in the dual pair settings by Tan-Zhu \cite{tan-zhu}. Nishiyama-Ochiai-Taniguchi  \cite{NOT} and Enright-Willenbring \cite{EW} independently gave characterizations of the GK dimensions and associated varieties for unitary highest weight HC modules appearing in the dual pair settings. In recent years, a uniform formula for the GK dimensions and associated varieties of unitary highest weight HC modules was discovered by Bai-Hunziker \cite{BH}. Almost at the same time, Barchini-Zierau \cite{BZ} and Zierau \cite{Zie18} computed the characteristic cycles of highest weight HC modules with regular integral infinitesimal characters, which can give some information of the corresponding associated varieties. In this paper, we determine associated varieties of the highest weight HC modules  for all the Hermitian types (Theorems 6.2 and 7.1), where the highest weights are not necessarily regular or integral and the modules are not necessarily unitary.  Note that the case of type $ A $ was already known in \cite{BX}.

\subsection{Lusztig's $ \aff $-functions}\label{subsec:intr2} 

There is a beautiful formula by Lusztig \cite{lusztig1985A_n} connecting the GK dimensions and the $ \aff $-functions. Using this formula, one can translate the computation of GK dimensions of simple highest weight modules to the computation of Lusztig's $ \aff $-functions on  the integral Weyl subgroups \cite{BX}. Thus, for our purposes we need to find a fine formula of Lusztig's  $ \aff $-functions first.

Fix a Cartan subalgebra $\mathfrak{h}$ of $\mathfrak{g}$.
Let $\Phi$ be the root system of $(\mathfrak{g}, \mathfrak{h})$ and $W$ be its Weyl group. The function $\aff: W\rightarrow\mathbb{N}$ first appeared in Lusztig's works on representations of finite groups of Lie type and Weyl groups  \cite{lusztig1977symbol, Lu79,Lusztig1982A-class-II, lusztig1984char}. It is originally defined as a function on the set of irreducible representations of Weyl groups. In  the study of cells of affine Weyl groups, Lusztig \cite{lusztig1985cellsI,lusztig1985A_n} introduced the $ \aff $-functions on Weyl groups defined in terms of Kazhdan-Lusztig basis of Hecke algebras. He found many interesting properties about $ \aff $-functions and cells. In fact, the $ \aff $-functions on the  set of simple modules and on the Weyl group are compatible \cite[Prop. 20.6]{lusztig2003hecke}.

It is well known that there is an algorithm of the $ \aff $-functions for Weyl groups of type $ A $, using Robinson-Schensted insertion. This gives rise to a combinatorial  algorithm of GK dimensions of simple highest weight modules of $\mathfrak{sl}(n,\mathbb{C})$ \cite[Thm. 4.6]{BX}. From this, we can determine the GK dimensions and the associated varieties of  highest weight HC modules for type $A$ (see \cite{BX} for more details). This paper is a sequel of \cite{BX}, and we will give an answer for other classical types.

Note that the Weyl groups of type $B$ and $C$ are the same. For Weyl groups of types $ B $ and $ D $, Lusztig gave a formula for the $ \aff $-functions on the set of irreducible representations of the Weyl groups \cite{lusztig1977symbol, Lu79,lusztig1984char}. Since the irreducible representations are parametrized by Lusztig's symbols \cite{lusztig1977symbol}, we need  a certain combinatorial correspondence between  the elements of Weyl groups and the irreducible representations. This was realized in the work \cite{Barbasch-vogan1982classical} by Barbasch and Vogan. In other words, they gave a way of associating an element of the Weyl group to a symbol. Combining these results together, we get combinatorial algorithms for computing $ \aff $-functions on the Weyl groups of types $ B $ and $ D $ (Propositions \ref{prop:afB} and \ref{prop:afD}). 
It is quite surprising that we can simplify these algorithms and get a uniform and efficient formula of $ \aff $-functions on classical Weyl groups (Theorem \ref{thm1}).

\subsection{A uniform formula}

Now we describe our formula in detail. Fix a positive integer $N$. We represent a \textit{partition} of $N$ as a decreasing sequence $p=(p_1, p_2, \cdots, p_N)$ of nonnegative integers. For simplicity, we often omit the trailing zeros in $p$. Each partition of $N$ corresponds a \textit{Young diagram}, which is an array of $N$ boxes having left-justified rows with the $i$-th row containing $p_i$ boxes for $1\leq i\leq N$. We will not distinguish a partition and its Young diagram. A \textit{Young tableau} $Y$ of shape $p$ is obtained by filling each box of the Young diagram $p$ with a number (which could be integral, real or complex, as required in context). In this case, we write $p=\sh(Y)$.

For a Young diagram $p$, use $ (k,l) $ to denote the box in the $ k $-th row and the $ l $-th column.
We say the box $ (k,l) $ is \textit{even} (resp. \textit{odd}) if $ k+l $ is even (resp. odd). Let $ p_i \ev$ (resp. $ p_i\od $) be the number of even (resp. odd) boxes in the $ i $-th row of the Young diagram $ p $. 
One can easily check that \begin{equation}\label{eq:ev-od}
	p_i\ev=\begin{cases}
		\left\lceil \frac{p_i}{2} \right\rceil&\text{ if } i \text{ is odd},\\
		\left\lfloor \frac{p_i}{2} \right\rfloor&\text{ if } i \text{ is even},
	\end{cases}
	\quad p_i\od=\begin{cases}
		\left\lfloor \frac{p_i}{2} \right\rfloor&\text{ if } i \text{ is odd},\\
		\left\lceil \frac{p_i}{2} \right\rceil&\text{ if } i \text{ is even}.
	\end{cases}
\end{equation}
Here for $ a\in \mathbb{R} $, $ \lfloor a \rfloor $ is the largest integer $ n $ such that $ n\leq a $, and $ \lceil a \rceil$ is the smallest integer $n$ such that $ n\geq a $. For convenience, we set
\begin{equation*}
p\ev=(p_1\ev,p_2\ev,\cdots)\quad\mbox{and}\quad p\od=(p_1\od,p_2\od,\cdots).
\end{equation*}


\begin{ex}
	Let $ p=(5,5,4,3,3) $. The even boxes in $p$ are marked as follows.
	\[
	\young(\times~\times~\times,~\times~\times~,\times~\times~,~\times~,\times~\times)
	\]
	Then $ p\ev=(3,2,2,1,2)$ and $ p\od=(2,3,2,2,1) $.
\end{ex}

For  a totally ordered set $ \Gamma $, we  denote by $ \mathrm{Seq}_n (\Gamma)$ the set of sequences $ x=(x_1,x_2,\cdots, x_n) $   of length $ n $ with $ x_i\in\Gamma $.
For the aim of this paper, it is enough to take $ \Gamma $ to be $ \mathbb{Z} $ or a coset of $ \mathbb{Z} $ in $ \mathbb{C} $. 

\begin{Defff}\label{def:FY}
	Applying the Robinson-Schensted algorithm to $ x\in \mathrm{Seq}_n (\Gamma)$, we  get a Young tableau $ Y(x) $. Denote $ p(x)=\sh(Y(x)) $.
\end{Defff}


For $ x=(x_1,x_2,\cdots,x_n)\in \seq $, set
\begin{equation*}
	\begin{aligned}
		{x}^-=&(x_1,x_2,\cdots,x_{n-1}, x_n,-x_n,-x_{n-1},\cdots,-x_2,-x_1),\\
		{}^-{x}=&(-x_n,-x_{n-1},\cdots, -x_2,-x_1,x_1,x_2,\cdots, x_{n-1}, x_n).
	\end{aligned}
\end{equation*}
Let $ W $ be a classical Weyl group. For $ w\in W $, denote by $\vv{w}$ the sequence $(w(1), w(2), \cdots, w(n))$.

\begin{Thm}\label{thm1}
	For an element $w$ of a classical Weyl group $ W$,
	\begin{equation*}\label{eq:afunc}
		\aff(w)=\begin{cases}
			\sum_{i\geq1}(i-1)p(\vv w)_i &\text{ if }\Phi=A_{n-1},\\
			\sum_{i\geq1}(i-1)p({}^-\vv{w})_i\od &\text{ if }\Phi=B_n,\\
			\sum_{i\geq1}(i-1)p({}^-\vv{w})_i\ev &\text{ if }\Phi=D_n.
		\end{cases}
	\end{equation*}
\end{Thm}

\begin{ex}
	Let $\Phi=B_5$ and $\vv w=(3, 4, -1, 5, 2)$. We have
	\[
	Y({}^-\vv{w})=\begin{tikzpicture}[scale=\domscale+0.1,baseline=-20pt]
		\hobox{0}{0}{-5}
		\hobox{1}{0}{-4}
		\hobox{2}{0}{-3}
		\hobox{3}{0}{-1}
		\hobox{4}{0}{2}
		\hobox{5}{0}{5}
		\hobox{0}{1}{-2}
		\hobox{1}{1}{1}
		\hobox{2}{1}{3}
		\hobox{3}{1}{4}		
	\end{tikzpicture}.
\]
Hence  $p({}^-\vv{w})\od=(3,2)$ and $\aff(w)=0\cdot3+1\cdot2=2$.
\end{ex}

\subsection{GK dimensions of highest weight modules}\label{subsec:gkh}

Based on  the above result for the $ \aff $-functions, we can deduce an efficient algorithm for the GK dimensions of simple highest weight modules for classical Lie algebras.  

Let $\Phi^+\subset\Phi$ be the set of positive roots determined by a Borel subalgebra $\mathfrak{b}$ of $\mathfrak{g}$. Denote by $\Delta$ the set of simple roots in $\Phi^+$. For $\lambda\in\mathfrak{h}^*$, define the \textit{Verma module}
\begin{equation*}
M(\lambda)=U(\mathfrak{g})\otimes_{U(\mathfrak{b})}\mathbb{C}_{\lambda-\rho},
\end{equation*}
where $\mathbb{C}_{\lambda-\rho}$ is a one-dimensional $\mathfrak{b}$-module with weight $\lambda-\rho$ and $\rho$ is  half the sum of positive roots. Denote by $L(\lambda)$ the simple quotient of $M(\lambda)$.


\begin{Thm}\label{thm2}
	Let $\lambda=(\lambda_1, \lambda_2, \cdots, \lambda_n)\in \mathfrak{h}^*$ be an integral weight. Then 
	\begin{equation*}
		\gkd L(\lambda)=\begin{cases}
			\frac{n(n-1)}{2}-\sum_{i\geq1}(i-1)p(\lambda)_i &\text{ if }\Phi=A_{n-1},\\
			n^2-\sum_{i\geq1}(i-1)p(\lambda^-)_i\od &\text{ if }\Phi=B_n/C_n,\\
			n^2-n-\sum_{i\geq1}(i-1)p(\lambda^-)_i\ev &\text{ if }\Phi=D_n.
		\end{cases}
	\end{equation*}
\end{Thm}
The case of type $A$ is already known (e.g., \cite{BX}). The case of types $B/C$, for which we need to consider upper closures of chambers, is a relatively straightforward consequence of Theorem \ref{thm1}. The case of type $D$, however, is much more difficult than we expected, for which we need to prove a critical property (Theorem \ref{thm:hollow}) about the domino tableaux in \cite{Garfinkle1}. To achieve this, a new invariant, the hollow tableau, is proposed in \S4. Van Leeuwen's result \cite{Leeuwen-elementary}, which is about the connection between Garfinkle's bumping description of domino tableaux and the left-right insertion description given in earlier work of Barbasch-Vogan \cite{Barbasch-vogan1982classical}, is also necessary.



\begin{ex}
	Let $\Phi=B_8$ and $\lambda=(3, 4, 1, -2, 0, -3, 5, 6)$. Then 
	\[
	Y(\lambda^-)=\begin{tikzpicture}[scale=\domscale+0.1,baseline=-45pt]
		\hobox{0}{0}{-6}
		\hobox{1}{0}{-5}
		\hobox{2}{0}{-4}
		\hobox{3}{0}{-3}
		\hobox{0}{1}{-3}
		\hobox{1}{1}{-1}
		\hobox{2}{1}{0}
		\hobox{3}{1}{2}	
		\hobox{0}{2}{-2}
		\hobox{1}{2}{0}
		\hobox{2}{2}{5}
		\hobox{3}{2}{6}
		\hobox{0}{3}{1}
		\hobox{1}{3}{3}
		\hobox{0}{4}{3}
		\hobox{1}{4}{4}	
	\end{tikzpicture}.
	\]
	Therefore, $p(\lambda^-)\od=(2, 2, 2, 1, 1)$ and 
	\[
	\gkd L(\lambda)=8^2-(0\cdot 2+1\cdot 2+ 2\cdot 2+ 3\cdot1+4\cdot1)=51.
	\]
\end{ex}

The non-integral case which needs more effort is given in Theorem \ref{thm:algo}.
 A main step is to decompose of the integral part $ \Phi_{[\lam]} $ of $ \Phi $ relative to $ \lam $ into some orthogonal subsystems. It is proved that these could be of type $ A $ or $ B $ if $ \Phi=B_n $;  $ A $, $ C$ or  $D $  if $ \Phi=C_n $; $ A $ or  $D $  if $ \Phi=D_n $.

\begin{Rem}
	It is an easy consequence that the $\aff$-functions and the GK dimensions can also be calculated by  column lengths of the Young diagrams, using Lemma \ref{def:F}. This is useful when the tableaux have many rows and much fewer columns.
\end{Rem}
\subsection{Associated varieties of highest weight modules}

Now we can compute GK dimensions of the highest weight HC modules for all Hermitian type Lie groups. For the classical types, we apply the above algorithm for GK dimensions of simple highest weight modules. For the exceptional types, we will need some results in \cite{CIS}. By Vogan \cite{Vogan-81} and Yamashita \cite{Hir01}, the associated varieties of such modules form a linear chain of closures of certain orbits, and hence they are determined by their GK dimensions. Combining these results, we can explicitly describe the associated varieties for all the highest weight HC modules, see Theorems \ref{thm:associated} and  \ref{67}.

\subsection{Organization}
This paper is organized as follows. In \S\ref{sec:pre}, we prepare some necessary preliminaries. In \S \ref{sec:af}, the formulas of Lusztig's $ \aff $-functions on classical Weyl groups are described. In \S\ref{sec:dom}, we prove a key result (Theorem \ref{thm:hollow}) about the domino tableaux. Then the GK dimensions of all the simple highest weight modules for classical Lie algebras are obtained in \S\ref{sec:lie}. In \S\ref{sec:HC} and \S\ref{sec:exp}, we pin down the associated varieties of the highest weight HC modules for all Hermitian type Lie groups.

\subsection*{Acknowledgments}
We would like to thank  G. Lusztig for helpful discussions about $\mathbf{a}$-functions. We are also very grateful to D. A. Vogan for his informative comments on $\aff$-functions and related topics. The first author is supported by NSFC Grant No. 12171344 and National Key R$\&$D Program of China (No. 2018YFA0701700 and No. 2018YFA0701701), the second author is supported by NSFC Grant No. 11701381 and the third author  is supported by
NSFC Grant No. 11801031 and 12171030.

%
%

\section{Preliminaries}\label{sec:pre}

%
%

In this section, we will give some brief preliminaries on GK dimensions, associated varieties and $\aff$-functions. See \cite{Vo78, Vo91, lusztig1984char} for more details.

\subsection{GK dimensions and associated varieties}

Let $M$ be a finite generated $U(\mathfrak{g})$-module. Fix a finite dimensional generating space $M_0$ of $M$. Let $U_{n}(\mathfrak{g})$ be the standard filtration of $U(\mathfrak{g})$. Set $M_n=U_n(\mathfrak{g})\cdot M_0$ and
\(
\text{gr} (M)=\bigoplus\limits_{n=0}^{\infty} \text{gr}_n M,
\)
where $\text{gr}_n M=M_n/{M_{n-1}}$. Thus $\text{gr}(M)$ is a graded module of $\text{gr}(U(\mathfrak{g}))\simeq S(\mathfrak{g})$.

\begin{Defff} The \textit{Gelfand-Kirillov dimension} of $M$  is defined by
\begin{equation*}
	\operatorname{GKdim} M = \overline{\lim\limits_{n\rightarrow \infty}}\frac{\log\dim( U_n(\mathfrak{g})M_{0} )}{\log n}.
\end{equation*}
\end{Defff}


\begin{Defff}
	The  \textit{associated variety} of $M$ is defined by
\begin{equation*}
	V(M):=\{X\in \mathfrak{g}^* \mid f(X)=0 \text{ for all~} f\in \operatorname{Ann}_{S(\mathfrak{g})}(\operatorname{gr} M)\}.
\end{equation*}
\end{Defff}

The above two definitions are independent of the choice of $M_0$, and $\dim V(M)=\gkd M$ (e.g., \cite{NOT}). If $M_0$ is $\mathfrak{a}$-invariant for a subalgebra $\mathfrak{a}\subset\mathfrak{g}$, then 
\begin{equation}\label{embed}
	V(M)\subset (\mathfrak{g}/\mathfrak{a})^*.
\end{equation}
\subsection{Associated varieties of Harish-Chandra modules} Let $G$ be a Lie group of  Hermitian type  with a maximal compact subgroup $K$. Denote by $\mathfrak{g}$ and $\mathfrak{k}$ their complexified Lie algebras respectively. Let $\mathfrak{g}=\mathfrak{k}\oplus\mathfrak{p}^+\oplus \mathfrak{p}^-$ be the usual decomposition of $\mathfrak{g}$ as  $\mathfrak{k}$-modules. The closures of  $K_{\mathbb{C}}$-orbits in $ \mathfrak{p}^+ $ form a linear chain of varieties \cite[Prop. 3.1]{Hir01},
\begin{equation*} 
	\{0\}={\overline{\mathcal{O}}}_0\subset \overline{\mathcal{O}}_1\subset ...\subset\overline{\mathcal{O}}_{r-1}\subset \overline{\mathcal{O}}_r=\mathfrak{p}^+,
\end{equation*}
where $r$ is the rank of the Hermitian symmetric space $G/K$ (see Table \ref{tab1}).

\begin{table}[t]
	
	\centering
	\renewcommand{\arraystretch}{1.5}
	\setlength\tabcolsep{5pt}
	\begin{tabular}{|l|l|l|l|}
		\hline
		$\operatorname{Lie}(G)$ &   $r$ & $c$ & $h^\vee-1$ \\  
		\hline  
		$\mathfrak{su}(k,n-k)$ & $\min\{k, n-k\}$ &  $1$ & $n-1$  \\ 
		$\mathfrak{sp}(n,\mathbb{R})$  & $n$ &   $1/2$ & $n$   \\ 
		$\mathfrak{so}^{*}(2n)$  & $\lfloor\frac{n}{2}\rfloor$  & $2$ & $2n-3$ \\ 
		$\mathfrak{so}(2,2n-1)$  & $2$  &  $n-3/2$ & $2n-2$ \\ 
		$\mathfrak{so}(2,2n-2)$  & $2$ &  $n-2$& $2n-3$ \\ 
		$\mathfrak{e}_{6(-14)}$  & $2$ &  $3$& $11$ \\ 
		$\mathfrak{e}_{7(-25)}$  & $3$ &  $4$& $17$ \\ 
		\hline
	\end{tabular}
	\bigskip
	\caption{Constants of  Lie groups of Hermitian type}
	\label{tab1}
\end{table}

We say the finite generated $U(\mathfrak{g})$-module $M$ is a \textit{Harish-Chandra module} if it has compatible actions of $\mathfrak{g}$ and $K$, and every $m\in M$ lies in a finite-dimensional $K$-invariant subspace, and every irreducible $K$-representation has finite multiplicity in $M$ (e.g., \cite{Vogan81book}). 

In particular, if $M=L(\lambda)$ is a highest weight HC-module, we can choose $M_0$ to be the finite dimensional $U(\mathfrak{k})$-module generated by $\mathbb{C}_{\lambda-\rho}$. Then $M_0$ is $\mathfrak{k}\oplus\mathfrak{p}^+$-invariant. In view of \eqref{embed},
\[
V(L(\lambda))\subset(\mathfrak{g}/(\mathfrak{k}\oplus\mathfrak{p}^+))^*\simeq(\mathfrak{p}^-)^*\simeq \mathfrak{p}^+,
\]
where the last isomorphism is induced from the Killing form. As shown in \cite{Vo91}, the associated variety $V(M)$ is also $K$-invariant. In fact, Yamashita \cite{Hir01} proved that $ V(M) $ must be one of $ \overline{\mathcal{O}}_{k} $.

\begin{Lem}
	Let $L(\lambda)$ be a highest weight HC-module. Then 
	\begin{equation*}
	V(L(\lambda))=\overline{\mathcal{O}}_{k(\lambda)}
	\end{equation*}
	for some $1\leq k(\lambda)\leq r$.
\end{Lem}

Bai-Hunziker \cite{BH} found a uniform expression for the dimension of $  \overline{\mathcal{O}}_k $.

\begin{Prop}\label{C: dimYk}
	For $0\leq k\leq r$, 
	\begin{equation*}\label{dimYk}
		\dim \overline{\mathcal{O}}_k=k(h^\vee-1) -k(k-1)c.
	\end{equation*}
The data of $r$, $h^\vee$ and $c$ are listed in Table \ref{tab1}.
\end{Prop}

\subsection{The \af-function}\label{sec:cell}

Recall that the Weyl group $ W  $ of $ \mf{g} $ is a Coxeter group generated by the set $ S=\{s_\al\mid\al\in\Delta \} $ of simple reflections. Let $\ell:W\to \mathbb{N}$ be the \textit{length function} on $W$. Given an indeterminate $v$, we have a Hecke algebra $ \mc{H} $ over $ \mc{A} :=\mathbb{Z}[v,v^{-1}]$. It has a basis $ T_w $, $ w\in W $ over $ \mc{A} $ with relations \[
T_wT_{w'}=T_{ww'} \text{ if }\ell(ww')=\ell(w)+\ell(w'),
\]\[
\text{and }(T_s+v^{-1})(T_s-v)=0 \text{ for any }s\in S.
\]
The Kazhdan-Lusztig basis $ C_w $, $ w\in W $  of $ \mc{H} $ are defined to be the unique elements $ C_w $ such that
\[
\overline{C_w}=C_w,\qquad C_w\equiv T_w \mod{\mc{H}_{<0}}
\]
where $ \bar{\,} :\mc{H}\rar\mc{H}$ is the bar involution such that $ \bar{q}=q^{-1} $ and $ \overline{T_w} =T_{w^{-1}}^{-1}$. Here $ \mc{H}_{<0}=\oplus_{w\in W}\mc{A}_{<0}T_w $ with $ \mc{A}_{<0}=v^{-1}\mathbb{Z}[v^{-1}] $. Let $ C_xC_y=\sum_{z\in W} h_{x,y,z}C_z $ with $ h_{x,y,x}\in\mc{A} $. Then $ \mathbf{a}:W\rar\mathbb{N} $ is defined by\[
\aff(z)=\max\{\deg h_{x,y,z}\mid x,y\in W \} \text{ for } z\in W.
\]

The following results can be found in \cite{lusztig2003hecke}.

\begin{Lem}
	\label{lem:Hecke}
	Let $ (W,S) $ be a Coxeter group.
	\begin{itemize}
		\item [(1)] For all $ w\in W $, $ \aff(w)=\aff(w^{-1}) $ .
		\item [(2)] If $ w_I $ is the longest element of the parabolic subgroup of $ W $ generated by a subset $ I\subset S $, then $ \aff(w_I)=\ell(w_I) $.
		\item [(3)] If $ W $ is a direct product of  Coxeter subgroups $ W_1 $ and $ W_2 $, then\[
		\aff(w)=\aff(w_1)+\aff(w_2)
		\]for  $ w=(w_1,w_2) \in W_1\times W_2=W$.
	\end{itemize}
\end{Lem}

\subsection{GK dimensions} Let $(-, -)$ be the standard bilinear form on $\mathfrak{h}^*$. For $\mu\in\mathfrak{h}^*$, define 
\begin{equation*}
\Phi_{[\mu]}:=\{\alpha\in\Phi\mid\langle\mu, \alpha^\vee\rangle\in\mathbb{Z}\},
\end{equation*}
where $\langle\mu, \alpha^\vee\rangle=2(\mu, \alpha)/(\alpha, \alpha)$. Set 
\[
W_{[\mu]}:=\{w\in W\mid w\mu-\mu\in \mathbb{Z}\Phi\}.
\]
Then $\Phi_{[\mu]}$ is a root system with Weyl group $W_{[\mu]}$. 
 Let $\Delta_{[\mu]}$ be the simple system of $\Phi_{[\mu]}$. Set $J=\{\alpha\in\Delta_{[\mu]}\mid\langle\mu, \alpha^\vee\rangle=0\}$. Denote by $W_J$ the Weyl group generated by reflections $s_\alpha$ with $\alpha\in J$. Let $\ell_{[\mu]}$ be the length function on $W_{[\mu]}$. Thus $\ell_{[\mu]}=\ell$ when $\mu$ is integral. Put
\begin{equation*}\label{ceq1}
	W_{[\mu]}^J:=\{w\in W_{[\mu]}\mid \ell_{[\mu]}(ws_\alpha)=\ell_{[\mu]}(w)+1\ \mbox{for all}\ \alpha\in J\}.
\end{equation*}
Thus $W_{[\mu]}^J$ consists of the shortest representatives of the cosets $wW_J$ with $ w\in W_{[\mu]} $. When $\mu$ is integral, we simply write $W^J:=W_{[\mu]}^J$ .

A weight $ \mu\in\hs $ is called \textit{anti-dominant} if $ \bil{\mu}{\al} \notin\mathbb{Z}_{>0}$ for all $ \al\in\Phi^+ $. For any $\lambda\in\mathfrak{h}^*$, there exists a unique anti-dominant weight $\mu\in\hs$ and a unique $w\in W_{[\mu]}^J$ such that $\lambda=w\mu$. 

\begin{Prop}[{\cite[Prop. 3.8]{BX}}]\label{pr:main1}
	Let $ \lam\in\hs $. Suppose that $\lambda=w\mu$, where $\mu$ is anti-dominant and $w\in W_{[\mu]}^J$. Then
\begin{equation*}
	\gkd L(\lambda)=|\Phi^+|-\aff_{[\lambda]}(w),
\end{equation*}
	where $\aff_{[\lambda]}$ is the $\aff$-function on $W_{[\lambda]}=W_{[\mu]}$.
\end{Prop}

\subsection{Upper closures of chambers} 
Let $E:=\mathbb{R}\Phi$. For $ \lambda\in E\subset\hs$, define 
\begin{equation}\label{eq:phi-lam}
	\Psi_\lambda^+=\{\alpha\in\Phi^+\mid\langle\lambda, \alpha^\vee\rangle>0\}.
\end{equation}
The set
\begin{equation*}
C_\circ:=\{\lambda\in E\mid\langle\lambda, \alpha^\vee\rangle<0\ \mbox{for all}\ \alpha\in\Phi^+\}.
\end{equation*}
is called the anti-dominant chamber in $ E $.
For any chamber $ C $, there exists a unique $ w\in W $ such that $C=wC_\circ$. Denote 
\begin{equation*}
\begin{aligned}
\Phi_C^+&=\{\alpha\in \Phi^+\mid\langle\lambda, \alpha^\vee\rangle>0\ \mbox{for all}\ \lambda\in C\},\\
\Phi_C^-&=\{\alpha\in \Phi^+\mid\langle\lambda, \alpha^\vee\rangle<0\ \mbox{for all}\ \lambda\in C\}.
\end{aligned}
\end{equation*}
Thus $\Phi^+=\Phi_C^+\sqcup \Phi_C^-$. Obviously $\Phi_C^+=\Psi_\lambda^+$ for any $\lambda\in C$.  

\begin{Defff}\label{defuc}
	If $C$ is a chamber of $E$, the \textit{upper closure} $\widehat C$ of $C$ is the set of $\lambda\in E$ such that $\langle\lambda, \alpha^\vee\rangle>0$ for $\alpha\in\Phi_C^+$ and $\langle\lambda, \alpha^\vee\rangle\leq 0$ for $\alpha\in\Phi_C^-$.
\end{Defff}

Here we collect some easy facts \cite{Hum90}.

\begin{Lem}\label{uclem1}
	
	\begin{itemize}
		\item [(1)] $E=\bigsqcup\limits_{w\in W} \widehat{wC_\circ}$.
		\item [(2)] $\Phi_{wC_\circ}^+=\{\alpha\in\Phi^+\mid w^{-1}\alpha<0\}$ and $\Phi_{wC_\circ}^-=\{\alpha\in\Phi^+\mid w^{-1}\alpha>0\}$.
	\end{itemize}
\end{Lem}

\begin{Lem}\label{uclem2}
	Let $w\in W$ and $\mu$ be an integral anti-dominant weight. Set $J=\{\alpha\in\Delta\mid\langle\mu, \alpha^\vee\rangle=0\}$. Then the following conditions are equivalent: $(1)$ $w\in W^J$; $(2)$ $w\gamma>0$ for all $\gamma\in J$; $(3)$ $w\mu\in\widehat{wC_\circ}$; $(4)$ $\Psi_{w\mu}^+=\Phi_{wC_\circ}^+$. 
\end{Lem}

\begin{proof}
	The equivalence of (1) and (2) is clear, while the equivalence of (3) and (4) follows from Definition \ref{defuc}. Set $\lambda=w\mu$ and $C=wC_\circ$. 
	
	Now we show (2) implies (3). Let $\alpha\in\Phi_C^+$. Lemma \ref{uclem1} yields $w^{-1}\alpha<0$. Set $\beta=-w^{-1}\alpha$. We get $\langle\mu, \beta^\vee\rangle\leq0$ since $\mu$ is anti-dominant. 
	The equality can not hold; otherwise, $\beta$ must be a sum of roots in $J$, and then (2) yields $-\alpha=w\beta>0$, a contradiction. Thus $\langle\mu, \beta^\vee\rangle<0$, and hence $\langle\lambda, \alpha^\vee\rangle>0$. On the other hand, if $\alpha\in\Phi_C^-$, then $\langle\lambda, \alpha^\vee\rangle=\langle\mu, (w^{-1}\alpha)^\vee\rangle\leq 0$. Then (3) is proved.
	
	It remains to show (3) implies (1). If $\lambda\in\widehat{wC_\circ}$, then we can find $w'\in W^J$ such that $\lambda=w'\mu=w\mu$. The previous argument implies $\lambda\in \widehat{w'C_\circ}$. This forces $w'=w$, in view of Lemma \ref{uclem1}(1).
\end{proof}

The following result is an easy consequence of Lemmas \ref{pr:main1} and \ref{uclem2}.

\begin{Lem}\label{ucgk}
	Let $C$ be a chamber of $E$. Suppose that $\lambda, \nu\in\widehat{C}$ are two integral weights. Then $\gkd L(\lambda)=\gkd L(\nu)$.
\end{Lem}

\subsection{Partitions and dual partitions} 
We say $q=(q_1, \cdots, q_N)$ is the \textit{dual partition} of a partition $p=(p_1, \cdots, p_N)$ and write $q={}^tp$ if $q_i$ is the length of $i$-th column of the Young diagram $p$. 
Recall that $p(x)$ is the shape of the Young tableau $Y(x)$ obtained by applying Robinson-Schensted algorithm to $x\in \mathrm{Seq}_n (\Gamma)$. For convenience, we set $q(x)={}^tp(x)$.

\begin{Lem}[\textbf{Definition}]\label{def:F}
	 Let $x\in \mathrm{Seq}_n (\Gamma)$. Then
	\begin{equation*}
	\begin{aligned}
		F_a(x):&=\sum_{i\geq1}(i-1)p_i=\sum_{i\geq1}\frac{q_i(q_i-1)}{2},\\
		F_b(x):&=\sum_{i\geq1}(i-1)p_i\od=\sum_{2\nmid i}(q_i\od)^2+\sum_{2\mid i}q_i\od(q_i\od-1),\\
		F_d(x):&=\sum_{i\geq1}(i-1)p_i\ev=\sum_{2\nmid i}q_i\ev(q_i\ev-1)+\sum_{2\mid i}(q_i\ev)^2,
	\end{aligned}
   \end{equation*}
where $p=p(x)=(p_1,p_2,\cdots, p_n) $ and $q=q(x)={}^tp(x)=(q_1, q_2, \cdots, q_n)$.
\end{Lem}

\begin{proof}
	Here we only prove the first equation since the ideas are similar. In fact, the first one follows from
	\[
	\sum_{i\geq 1}(i-1)p_{i}=\sum_{i\geq 1}\sum_{j=1}^{p_{i}}(i-1)=\sum_{j\geq 1}\sum_{i=1}^{q_j}(i-1)=\sum_{j\geq1}\frac{q_j(q_j-1)}{2}.
	\]
\end{proof}

\begin{ex}\label{eq:F1}
	If $ x=(x_{1}) $ is a sequence of length 1,  we  have 
	\begin{equation*}
		F_a(x)=F_d(x^-) =0,\quad F_b(x^-)=\begin{cases}
			0&\text{ if }x_{1}\leq0\\1 &\text{ if }x_{1}>0
		\end{cases} .
	\end{equation*}
\end{ex}


The following lemma is useful.
\begin{Lem}\label{RSeq}
	Let $x=(x_1, \cdots, x_n)\in\mathrm{Seq}_n (\Gamma)$. Assume that $y\in\mathrm{Seq}_n (\Gamma)$ such that for all $1\leq i<j\leq n$, $y_i<y_j$ whenever $x_i\leq x_j$ and $y_i>y_j$ whenever $x_i>x_j$. Then $p(x)=p(y)$.
\end{Lem}

For $x\in \mathrm{Seq}_n (\Gamma)$ and $k\leq n$, denote by $c_k(x)$ (resp. $d_k(x)$) the maximal length of
subsequences which can be written as a union of $k$  \emph{weakly} increasing (resp. \emph{strictly} decreasing) subsequences of $x$. The following result is an immediate  consequence of \cite{greene1974} and Lemma \ref{RSeq}.

\begin{Thm}[Greene's Theorem]\label{gthm}
	Let $x\in \mathrm{Seq}_n (\Gamma)$ and $k\leq n$. Then
	\begin{equation*}
	c_k(x)=p_1+p_2+...+p_k\mbox{ and } d_k(x)=q_1+q_2+...+q_k,
	\end{equation*}
	where $p=p(x)=(p_1,p_2,\cdots, p_n)$ and $q=q(x)={}^tp(x)=(q_1,q_2,\cdots, q_n)$.
\end{Thm}

%
%

\section{Formulas for $ \aff $-functions}\label{sec:af}

%
%

In this section, we will prove Theorem \ref{thm1} about formulas for $\aff$-functions. The formula for type $A$ is well-known (Proposition \ref{prop:bA}). We will deal mainly with the cases of types $B$ (Proposition \ref{prop:bB}) and $D$ (Proposition \ref{prop:bd}).

For $i, j\in\mathbb{Z}$ with $i<j$, set $[i,j]= \{i,i+1,\cdots, j-1, j\} $. Recall that we map $w\in W$ for classical types to the sequence $\vv w=(w(1),\cdots,w(n)) $. Denote by $s_i$ the involution $ w\in W$ such that $ w(i)=i+1 $, $ w(i+1)=i$ and $w(k)=k$ for all $k\in[1, i-1]\cup[i+2, n]$.


\subsection{The system $A_{n-1}$}\label{subsec:an}

In this case, $W$ is the symmetric group $ \mf{S} _n$ of $ n $ letters and $S=\{s_1, \cdots, s_{n-1}\}$ is the set of adjacent transpositions.

\begin{Prop}\label{prop:bA}
	Let $\Phi=A_{n-1}$ and $w\in W$. Then
	\[
	\aff(w)=\sum_{i\geq 1}(i-1)p(\vv w)_i=F_a(\vv w).
	\]
\end{Prop}
\begin{proof}
	Fix $w$ and set $q=q(\vv w)={}^tp(\vv w)$. In view of \cite{BX}, we have $\aff(w)=\sum_{i\geq 1} \frac{1}{2}q_i(q_i-1)$. Then the proposition follows from Lemma \ref{def:F}.
\end{proof}

\subsection{The system $B_n$ and $B$-symbols}\label{subsec:bn}
In this case, $W=W_n$, where $ W_n $ is the group consisting of permutations $ w $ of the set $[-n, n]$
such that $ w(-i)=-w(i) $ for all $ i \in[1,n]$. Let $ t\in W_n$ be the element with $\vv t=(-1,2, \cdots, n)$. Then $ (W_n , S_n)$ is a Coxeter system, with $S_n=\{t,s_1,\cdots s_{n-1}\} $. If $ n=1 $,  $ W_1\simeq\mf{S}_2$ is the Weyl group of type $A_1$; if $ n\geq 2$,  $W_n$ is the Weyl group of type $ B_n$. For convenience, we set $ B_1=A_1 $.

Now we recall the Lusztig's symbols \cite{lusztig1977symbol}. Let 
\[ 
\begin{pmatrix}
\lambda_1~\lambda_2~\cdots~\lambda_{m+1}\\\mu_1~\mu_2~\cdots~\mu_m
\end{pmatrix}, m\geq 0
\]
be a tableau of nonnegative integers such that entries in each row are  strictly increasing. Define an equivalence relation on the set of all such tableaux via
\begin{equation*}\label{eq:equiv}
\begin{pmatrix}
\lambda_1~\lambda_2~\cdots~\lambda_{m+1}\\\mu_1~\mu_2~\cdots\mu_m
\end{pmatrix} \sim
\begin{pmatrix}
0~\lambda_1+1~\lambda_2+1~\cdots~\lambda_{m+1}+1\\0~\mu_1+1~\mu_2+1~\cdots\mu_m+1
\end{pmatrix} .
\end{equation*}
Denote by $ \Sigma_B $ the set of equivalence classes under $ \sim $. We use the same notation $ \Lambda=\begin{pmatrix}
\lambda_1~\lambda_2~\cdots~\lambda_{m+1}\\\mu_1~\mu_2~\cdots\mu_m
\end{pmatrix}  \in\Sigma_B$ to denote its equivalence class, called a \textit{$ B $-symbol}.
There is a well-defined function 
\begin{equation}\label{eq:a-syb}
\cff_B:\Sigma_B\to \mathbb{N}
\end{equation}
such that 
\begin{multline*}
\cff_B(\Lambda)=\sum_{1\leq i<j\leq m+1}\min\{\lambda_i,\lambda_j\}+\sum_{1\leq i<j\leq m}\min\{\mu_i,\mu_j\}\\
+\sum_{\substack{1\leq i\leq m+1\\1\leq j\leq m}}\min\{\lambda_i,\mu_j\}-\frac16 m(m-1)(4m+1).
\end{multline*}
For $ w\in W_n $, recall that  $ Y({}^-{\vv w}) $ is the Young tableau obtained by applying the Robinson-Schensted algorithm to the sequence\[
{}^-\vv{w}=(-w(n),-w(n-1),\cdots, -w(1), w(1), \cdots, w(n-1), w(n)),
\]
and $p({}^-\vv{w})=\sh(Y({}^-\vv{w}))$, which is a partition of $2n$. By \cite[Prop.17]{Barbasch-vogan1982classical}, there is a symbol
 $\Lambda =\begin{pmatrix}
\lambda_1~\lambda_2~\cdots~\lambda_{m+1}\\\mu_1~\mu_2~\cdots\mu_m
\end{pmatrix} \in\Sigma_B $ such that, as multisets
\begin{equation}\label{eq:symb}
  \{2\lambda_i,2 \mu_j +1\mid  i\leq  m+1, j\leq m  \} =\{ p_k+2m+1-k \mid  k\leq  2m+1\},
\end{equation}
where $p=p({}^-\vv{w})=(p_1,p_2,\cdots)$ with $p_{2m+2}=0$. Here we set $p_l=0$ when $l>2n$. It gives a well-defined map\begin{equation}\label{eq:mapsymb}
Symb_B:W_n\to \Sigma_B.
\end{equation}

\begin{Prop}\label{prop:afB}
	The $ \aff $-function on $W_n$ is the composition $\cff _B \circ Symb_B$ of maps from \eqref{eq:a-syb} and \eqref{eq:mapsymb}, that is, $\aff(w)=\cff _B \circ Symb_B(w)$ for all $w\in W_n$.
\end{Prop}

This proposition is implied in the works of Barbasch-Vogan and Lusztig. Since we have not found a straightforward
statement about it in the literature, we provide a proof here. 

\begin{proof}
	\newcommand{\irr}{\operatorname{Irr} }
	Here $W=W_n$. Write $\cff=\cff_B$, $\Sigma=\Sigma_B$ and $Symb=Symb_B$. Let $ \irr(W)$ be the set of isomorphism classes of simple $ W$-modules over $ \mathbb{C} $. Consider the two-sided cells defined in \cite{KL}, which are the equivalence classes of a relation ``$\sim_{LR}$'' on $W$. By \cite[\S20.2]{lusztig2003hecke} (or \cite[\S5.1]{lusztig1984char}), each simple $W$-module $E\in \irr(W)$ corresponds to a unique two-sided cell $\eff$ of $W$. In this case, Lusztig wrote $ E\sim_{LR} x $ for every $ x\in \eff $. By \cite[\S20.6]{lusztig2003hecke}, there is a function 
	\[
	\aff: \irr(W)\to\mathbb{N},\quad E\mapsto \aff_E,
	\]
 such that $\aff_E=\aff(x)$ whenever $E\sim_{LR} x$. In \cite[\S4.5]{lusztig1984char}, Lusztig defined a bijection 
	\[
	Symb':\irr(W)\to \Sigma,
	\]
	and proved that $\aff_E=\cff(Symb'(E))$ for every $E\in\irr(W)$. 
	
	On the other hand, it was shown in \cite[Prop. 17, Thm.18, Lem.29]{Barbasch-vogan1982classical} that $ E\sim_{LR}x $ if and only if $ Symb'(E) $ is a permutation of $ Symb(x) $. For each $x\in W$, take $ E\in\irr(W) $ such that  $ Symb'(E)=Symb(x) $. Then $E\sim_{LR}x$, and  
$ 	\aff(x)=\aff_E=\cff(Symb'(E))=\cff(Symb(x)). $
\end{proof}


\begin{Prop}\label{prop:bB}
	Use notation in \S 1.3 and \S2.6. For $w\in W_n$,
\begin{equation*}
\aff(w)
=\sum_{ k\geq 1}  (k-1)p_k\od=F_b({}^-\vv{w}),
\end{equation*}
where $p=p({}^-\vv{w})=(p_1, p_2, \cdots)$.
\end{Prop}

\begin{proof}
	Assume that $p_{2m+2}=0$ and
	\[
	Symb_B(w)=\begin{pmatrix}
	\lambda_1~\lambda_2~\cdots~\lambda_{m+1}\\\mu_1~\mu_2~\cdots\mu_m
	\end{pmatrix} .
	\]
	Let $ \alpha_k=\left\lfloor \frac{p_{k}+2m+1-k  }{2}\right\rfloor $ for $1\leq k\leq 2m+1$. In view of \eqref{eq:symb},
	\[
	\{\lambda_i,\mu_j \mid  1\leq  i\leq m+1, 1\leq j \leq m \}  = \left\{ \alpha_k\mid 1\leq k\leq 2m+1\right \} .
	\]
It follows from Proposition \ref{prop:afB} that
\begin{equation*}
 \aff(w) = \sum_{1\leq k_1<k_2\leq 2m+1} \min \left\{\alpha_{k_1},\alpha_{k_2}\right\}- \frac{1}{6}m(m-1)(4m+1).
\end{equation*}
Since  $ \alpha_k$ is  decreasing with respect to $k$,
\[
\aff(w) =
\sum_{1\leq k\leq 2m+1} (k-1)\alpha_k- \frac{1}{6}m(m-1)(4m+1).
\]
By \eqref{eq:ev-od}, one has $ \alpha_k=p_k\od+ \left\lfloor \frac{2m+1-k  }{2}\right\rfloor $. Then the lemma follows from
\[
\sum_{ k=1}^{ 2m+1} (k-1)\left\lfloor \frac{2m+1-k  }{2}\right\rfloor=\sum_{i=1}^{ m-1} (m-i)(4i-1)= \frac{1}{6}m(m-1)(4m+1).
\]
\end{proof}

\begin{Rem}
One can easily see that $\aff(w)$ is independent of $m$, as long as $p_{2m+2}=0$, which can be easily achieved if we set for example $m=n$.	
\end{Rem}

\subsection{The system  $D_n$ and $D$-symbols}\label{subsec:dn}

In this case, $W=W_n'$, which consists of those elements $ w$ of  $W_ n$  such that  the number of negative integers in $ \{w(1), w(2), \cdots, w(n) \}   $ is even.  Let $ u=ts_1t $ and $ S'_n=\{u,s_1,s_2,\cdots,$  $s_{n-1}\} $. Thus $ (W_n',S'_n) $ is a Coxeter system. If $ n=1 $, then $ W_n' $ is trivial; if $ n=2 $ (resp. $ n=3 $, $ n\geq 4$), $ W_n' $ is a Weyl group of type $ A_1\times A_1 $ (resp. $ A_3 $, $ D_n $).  For convenience, we set $ D_2=A_1\times A_1 $, $ D_3=A_3 $.


Here we need Lusztig's $D$-symbols \cite{lusztig1977symbol}. Let 
\[ \begin{pmatrix}
\lambda_1~\lambda_2~\cdots~\lambda_{m}\\\mu_1~\mu_2~\cdots~\mu_m
\end{pmatrix}, m\geq 0\]
be a tableau of nonnegative integers such that entries in each row are  strictly increasing. Define an equivalence relation on the set of all such tableaux by
\begin{equation*}\label{eq:equiv-d}
\begin{pmatrix}
	\mu_1~\mu_2~\cdots\mu_m\\
	\lambda_1~\lambda_2~\cdots~\lambda_{m}
\end{pmatrix} \sim
\begin{pmatrix}
	\lambda_1~\lambda_2~\cdots~\lambda_{m}\\\mu_1~\mu_2~\cdots\mu_m
\end{pmatrix}
\sim
\begin{pmatrix}
0~\lambda_1+1~\lambda_2+1~\cdots~\lambda_{m}+1\\0~\mu_1+1~\mu_2+1~\cdots\mu_m+1
\end{pmatrix} .
\end{equation*}
Denote by $ \Sigma_D $ the set of equivalence classes under this relation $ \sim $. Use the same notation $ \Lambda=\begin{pmatrix}
\lambda_1~\lambda_2~\cdots~\lambda_{m}\\\mu_1~\mu_2~\cdots\mu_m
\end{pmatrix}  \in\Sigma_D$ to denote its equivalence class,  called a \textit{$ D$-symbol}. Define a map
\begin{equation}\label{eq:mapbd}
d:\Sigma_B\to \Sigma_D ,
\end{equation}
\[
\begin{pmatrix}
\lambda_1~\lambda_2~\cdots~\lambda_{m+1}\\\mu_1~\mu_2~\cdots\mu_m
\end{pmatrix}  \mapsto \begin{pmatrix}
\lambda_1&\lambda_2&\cdots&\lambda_{m+1}\\0 &\mu_1+1&\cdots&\mu_m+1
\end{pmatrix}
\]
and a function \begin{equation}\label{eq:af-symb-d}
\cff_D:\Sigma_D\to \mathbb{N}
\end{equation}
given by
\begin{multline*}
\cff_D(\Lambda)=\sum_{1\leq i<j\leq m}\min\{\lambda_i,\lambda_j\}+\sum_{1\leq i<j\leq m}\min\{\mu_i,\mu_j\}\\
+\sum_{1\leq i,j\leq m}\min\{\lambda_i,\mu_j\}-\frac16 m(m-1)(4m-5),
\end{multline*}
where $ \Lambda=\begin{pmatrix}
\lambda_1~\lambda_2~\cdots~\lambda_{m}\\\mu_1~\mu_2~\cdots\mu_m
\end{pmatrix}$. 

Let \begin{equation}\label{eq:mapsymb-d}
Symb_D: W'_n\to \Sigma_D
\end{equation}
be  the restriction to $ W_n' $ of the composition $ d\circ Symb_B $ of maps from  \eqref{eq:mapsymb} and \eqref{eq:mapbd}.

\begin{Prop}\label{prop:afD}
	The $ \aff $-function on $W'_n$ is the composition $ \cff _D\circ Symb_D$ of maps from \eqref{eq:af-symb-d} and \eqref{eq:mapsymb-d}, that is, $\aff(w)=\cff _D\circ Symb_D(w)$ for $w\in W'_n$. 
\end{Prop}

\begin{proof}
	We can simply imitate the proof of Proposition \ref{prop:afB}.
\end{proof}

\begin{Prop}\label{prop:bd}
	Use notation in \S1.3 and \S2.6. For $w\in W'_n$,
\begin{equation*}
\aff(w)=\sum_{k\geq1}(k-1)p_k\ev=F_d({}^-\vv{w}),
\end{equation*}
where $p=p({}^-\vv{w})=(p_1, p_2, \cdots)$.
\end{Prop}

\begin{proof}

Assume that $p_{2m+1}=0$ and
\[
Symb_B(w)= \begin{pmatrix}
	\lambda_1~\lambda_2~\cdots~\lambda_{m+1}\\\mu_1~\mu_2~\cdots\mu_m
\end{pmatrix}.
\]
We have
\[
\{2\lambda_i,2 \mu_j +1\mid   i\leq m+1,  j \leq m \}  =\{0\}\cup \{ p_k+2m+1-k \mid  k\leq 2m \}. 
\]
Thus $ \lambda_1=0 $  and $ \lambda_i\geq 1 $ for $ i>1$. The equivalence relation yields 
\[
Symb_D(w)= \begin{pmatrix}
\lambda_1&\lambda_2&\cdots&\lambda_{m+1}\\0 &\mu_1+1&\cdots&\mu_m+1
\end{pmatrix}= \begin{pmatrix}
\lambda_2-1&\cdots&\lambda_{m+1}-1\\\mu_1&\cdots&\mu_m
\end{pmatrix}.
\]
If $ p_k+2m+1-k $ is even, then $ {p_k+2m+1-k} = 2\lambda_i$ for some $ i $, and hence $ \lambda_i-1=\left\lfloor \frac{p_k+2m-k}{2} \right\rfloor $. If $ p_k+2m+1-k $ is odd, then $ {p_k+2m+1-k} = 2\mu_j+1$ for some $ j $, and hence $ \mu_j= \left\lfloor \frac{p_k+2m-k}{2} \right\rfloor$. Therefore, 
\[
\{\lambda_i-1,\mu_j\mid 2\leq i\leq m+1,1\leq j\leq m \}=\left\{ \beta_{k}| \,1\leq k\leq 2m  \right \},
\]
where $ \beta_k=\left\lfloor \frac{p_k+2m-k}{2} \right\rfloor  $.
It follows from Proposition \ref{prop:afD} that 
  \begin{align*}
  \aff(w)&=\sum_{1\leq k_1<k_2\leq 2m} \min \left\{\beta_{k_1},\beta_{k_2} \right\} - \frac{1}{6}m(m-1)(4m-5)\\
  &=\sum_{k=1}^{2m}(k-1) \beta_k- \frac{1}{6}m(m-1)(4m-5).
  \end{align*} 
With \eqref{eq:ev-od}, one has $ \beta_k=p_k\ev+ \left\lfloor \frac{2m-k  }{2}\right\rfloor $. A quick calculation yields the lemma.
\end{proof}

%
%

\section{Combinatorics of Dominoes}\label{sec:dom}

%
%

In this section, we introduce a new invariant, the hollow tableau. It will be used to prove a key result (Theorem \ref{thm:hollow}), which is necessary in the proof of Theorem \ref{thm2}. First we need to recall the domino insertion algorithm of Garfinkle \cite{Garfinkle1}.

\subsection{Domino insertion}

A \textit{domino} is a skew shape consisting of two adjacent boxes. 
A \textit{(standard) domino tableau} $D$ consists of a tiling of the Young diagram by dominoes and a filling of each domino with a positive integer, such that the numbers are strictly increasing along each row and column. If  $ D $ has shape $ p $, we write $ p=\sh(D) $. We say $D$ is a \textit{skew} domino tableau if $D=Y/Z$ for two domino tableaux $Z\subset Y$.

\begin{ex}
There are 3 possible domino tableaux of shape $ (4,2) $ filled with $ \{1,2,3\} $:\[
\begin{tikzpicture}[scale=\domscale,baseline=-15pt]
\hdom{0}{0}{1}
\hdom{2}{0}{2}
\hdom{0}{1}{3}
\end{tikzpicture} 
\quad
\begin{tikzpicture}[scale=\domscale,baseline=-15pt]
\hdom{0}{0}{1}
\hdom{2}{0}{3}
\hdom{0}{1}{2}
\end{tikzpicture} 
\quad
\begin{tikzpicture}[scale=\domscale,baseline=-15pt]
\vdom{0}{0}{1}
\vdom{1}{0}{2}
\hdom{2}{0}{3}
\end{tikzpicture} 
\]
\end{ex}

For any positive integer $j$, denote by $D_{\leq j}$ (resp. $D_{\geq j}$) the sub-tableau of $D$ which consists of those dominoes filled with integers less (resp. greater) than or equal to $j$. Obviously $D_{\leq j}$ is a  domino tableau. If $j$  occurs in $D$, we write $\dom_j(D)$ for the domino in $ D $ filled with  $j$. 

Now recall the domino insertion algorithm, which was introduced by Garfinkle \cite{Garfinkle1}. Here we follow the  description in \cite[\S3.2]{bonnafe-geck-iancu-lam2010}.

\begin{Defff}[Domino insertion algorithm]\label{def:dom}
Let $D$ be a domino tableau with shape $p$ and $i$ be a positive integer which does not appear in $D$. Now we define a new domino tableau $E=D\leftarrow ci$  with $c=\pm1$ as follows.

\begin{itemize}
	\item [(1)] Let $E_{\leq i}$ be the domino tableau obtained from $D_{\leq i-1}$ by adding a horizontal (resp. vertical) domino with value $i$ in the first row (resp. column) when $c=1$ (resp. $c=-1$).
	
	\item [(2)] For $ j> i $, assume that we have obtained $ E_{\leq j-1} $ with shape $ q_{j-1} $. If $D$ contains no domino labeled $j$, then set $E_{\leq j}=E_{\leq j-1}$, otherwise $ E_{\leq j} $ is a domino tableau obtained from $ E_{\leq j-1} $ by bumping the domino $ \dom_j $ in $ D $ occupied by $ j $ in the following way.
	\begin{itemize}
		\item [(i)] If $ q_{j-1}\cap   \dom_j =\emptyset$, then $ E_{\leq j} =E_{\leq j-1}\cup \dom_j   $.
		
		\item [(ii)] If $q_{j-1}\cap\dom_j$ is precisely a box $ (k,l) $, then
		we add a domino containing $j$ to $E_{\leq j-1}$ to obtain  $E_{\leq j}$ which has
		shape $q_{j-1}\cup\dom_j\cup(k+1, l+1)$.
		
		\item [(iii)] If $q_{j-1}\cap\dom_j=\dom_j$ and $\dom_j$ is horizontal (resp. vertical), then we bump the domino $\dom_j$ to the next row (resp. column), and set $E_{\leq j}$ to be the union of $E_{\leq j-1}$ with a horizontal (resp. vertical) domino with value $j$ adding in the next row (resp. column) of $\dom_j$.
	\end{itemize}
\item [(3)] Repeat  (2) until $ D_{\geq j+1} $ is empty, then set $ E=E_{\leq j} $.
\end{itemize}
\end{Defff}

\begin{ex}\label{ex:dom-insertion}
	Let $ E= D\leftarrow -1$ with \[
D=
	\begin{tikzpicture}[scale=\domscale,baseline=-25pt]
	\hdom{0}{0}{2}
	\hdom{2}{0}{5}
	\vdom{1}{1}{4}
	\vdom{0}{1}{3}
\hdom{2}{1}{6}
	\end{tikzpicture} .
	\]
	Then we have 
	
	\[
\begin{tikzpicture}[scale=\domscale,baseline=-25pt]
	\vdom{0}{0}{1}
\end{tikzpicture} 
\to
\begin{tikzpicture}[scale=\domscale,baseline=-25pt]
	\vdom{0}{0}{1}
	\vdom{1}{0}{2}
	\end{tikzpicture} 
\to
\begin{tikzpicture}[scale=\domscale,baseline=-25pt]
\vdom{0}{0}{1}
\vdom{1}{0}{2}
	\hdom{0}{2}{3}
\end{tikzpicture} 
\to
\begin{tikzpicture}[scale=\domscale,baseline=-25pt]
\vdom{0}{0}{1}
\vdom{1}{0}{2}
\hdom{0}{2}{3}
	\vdom{2}{0}{4}
\end{tikzpicture} 
\to
\begin{tikzpicture}[scale=\domscale,baseline=-25pt]
\vdom{0}{0}{1}
\vdom{1}{0}{2}
\hdom{0}{2}{3}
\vdom{2}{0}{4}
\vdom{3}{0}{5}
\end{tikzpicture}
\to
\begin{tikzpicture}[scale=\domscale,baseline=-25pt]
\vdom{0}{0}{1}
\vdom{1}{0}{2}
\hdom{0}{2}{3}
\vdom{2}{0}{4}
\vdom{3}{0}{5}
\hdom{2}{2}{6}
\end{tikzpicture}=E .
	\]
\end{ex}

\newcommand{\laa}{\leftarrow}

For $w\in W_n$, define a domino tableau as
\begin{equation*}
	P(w):=((\cdots((\emptyset\laa w(1))\laa w(2))\cdots)\laa w(n)).
\end{equation*}
The sequence of shapes obtained in the process gives rise to another  domino tableau $Q(w)$ which is called  the \textit{recording tableau}. Let $x=(x_1,x_2,\cdots x_n)$ be a sequence of nonzero integers with distinct absolute values. By a little abuse of notation, we still denote by $P(x)$ the domino tableau corresponding to the sequence $x$ and by $Q(x)$ the recording domino tableau.


\newcommand{\ra}{\to}

\begin{ex}\label{dex1}
	Let  $\vv w=(-3, -4, 1, 5, -2)\in W_5$. Then 
	\[
	\begin{tikzpicture}[scale=\domscale,baseline=-25pt]
	\vdom{0}{0}{3}
	\end{tikzpicture} 
	\ra
	\begin{tikzpicture}[scale=\domscale,baseline=-25pt]
	\vdom{0}{0}{3}
	\vdom{0}{2}{4}
	\end{tikzpicture} 
	\ra
	\begin{tikzpicture}[scale=\domscale,baseline=-25pt]
	\hdom{0}{0}{1}
	\hdom{0}{1}{3}
	\vdom{0}{2}{4}
	\end{tikzpicture} 
	\ra
	\begin{tikzpicture}[scale=\domscale,baseline=-25pt]
	\hdom{0}{0}{1}
	\hdom{0}{1}{3}
	\vdom{0}{2}{4}
	\hdom{2}{0}{5}
	\end{tikzpicture} 
	\ra
	\begin{tikzpicture}[scale=\domscale,baseline=-25pt]
	\hdom{0}{0}{1}
	\vdom{0}{1}{2}
	\vdom{1}{1}{3}
	\hdom{0}{3}{4}
	\hdom{2}{0}{5}
	\end{tikzpicture}=P(w),
	\]
	\bigskip	
	\[
	\begin{tikzpicture}[scale=\domscale,baseline=-25pt]
	\vdom{0}{0}{1}
	\end{tikzpicture} 
	\ra
	\begin{tikzpicture}[scale=\domscale,baseline=-25pt]
	\vdom{0}{0}{1}
	\vdom{0}{2}{2}
	\end{tikzpicture} 
	\ra
	\begin{tikzpicture}[scale=\domscale,baseline=-25pt]
	\vdom{0}{0}{1}
	\vdom{0}{2}{2}
	\vdom{1}{0}{3}
	\end{tikzpicture} 
	\ra
	\begin{tikzpicture}[scale=\domscale,baseline=-25pt]
	\vdom{0}{0}{1}
	\vdom{0}{2}{2}
	\vdom{1}{0}{3}
	\hdom{2}{0}{4}
	\end{tikzpicture} 
	\ra
	\begin{tikzpicture}[scale=\domscale,baseline=-25pt]
	\vdom{0}{0}{1}
	\vdom{0}{2}{2}
	\vdom{1}{0}{3}
	\hdom{2}{0}{4}
	\vdom{1}{2}{5}
	\end{tikzpicture}=Q(w).
	\]	
\end{ex}

The following result, which is due to Garfinkle \cite{Garfinkle1}, is a generalization of the Robinson–Schensted correspondence (e.g., \cite[Thm. 3.1.1]{Sagan}).

\begin{Thm}\label{dthmis}
	The assignment $ w\mapsto (P(w), Q(w)) $ gives rise to a bijection between $W_n$ and the set of pairs of  domino tableaux of the same shapes filled with $ \{1,2,\cdots,n\} $.
	 Moreover, 
	 \begin{equation}\label{eq:pqinv}
	 P(w)=Q(w^{-1}).
	 \end{equation}
\end{Thm}

The following result shows that the $\aff$-functions can also be calculated by domino algorithm (keeping in mind of Theorem \ref{thm1}).

\begin{Prop}\label{prop:shape}
	Let $w\in W_n$. Then $p({}^-\vv{w})=\sh(P(w))=\sh(Q(w))$.
\end{Prop}

Although it had been observed in earlier work of Garfinkle \cite{Garfinkle1}, a complete proof of this proposition was given in \cite{Leeuwen-elementary}. 

 
%

\subsection{Hollow tableaux} Now we introduce the hollow tableau.

\begin{Defff}
	By removing all the odd boxes from a domino tableau $ D $, we obtain a tableau $ \mc{H}(D) $  consisting of only even boxes and inheriting the filling from $ D $. We say $ \mc{H}(D) $ is the \textit{hollow tableau} of $ D $. 
\end{Defff}

\begin{ex}
	When $ D $ is the tableau in Example \ref{ex:dom-insertion}, we have
	\[
	D=
	\begin{tikzpicture}[scale=\domscale,baseline=-25pt]
		\hdom{0}{0}{2}
		\hdom{2}{0}{5}
		\vdom{1}{1}{4}
		\vdom{0}{1}{3}
		\hdom{2}{1}{6}
	\end{tikzpicture}
	\quad\ra\quad
	\begin{tikzpicture}[scale=\domscale,baseline=-25pt]
		\hobox{0}{0}{2}
		\hobox{2}{0}{5}
		\hobox{1}{1}{4}
		\hobox{3}{1}{6}
		\hobox{0}{2}{3}
	\end{tikzpicture}=\mathcal{H}(D).
	\]
	
\end{ex}

The following lemma follows from the definition.
\begin{Lem}\label{lem1:leftup}
	If $ (k,l) $ is a box filled in $ \mc{H}(D) $, then all the even boxes $ (a,b) $ are filled for all $ 1\leq a\leq k $, $ 1\leq b\leq l $
\end{Lem}

This subsection is devoted to proving the following theorem, which is interesting in its own right and will play a significant role in determining the GK dimensions of highest weight modules for type $ D $.


\begin{Thm}\label{thm:hollow}
	Let $ w\in W_n $. Then \[
	\mc{H}(P(w))=	\mc{H}(P(tw))=\mc{H}(P(wt)), 
	\]
	\[
	 \mc{H}(Q(w))=	\mc{H}(Q(tw))=\mc{H}(Q(wt)).
	\]
	In particular, $ p({}^-\vv{w})\ev=p({}^-\vv{tw})\ev =p({}^-\vv{wt})\ev $.  
\end{Thm}

\begin{Rem}\label{t}
	It can be easily verified that $(tw)(k)=-w(k)$ when $|w(k)|=1$ and $(tw)(k)=w(k)$ when $|w(k)|>1$. Similarly, $(wt)(1)=-w(1)$ and $(wt)(k)=w(k)$ when $k>1$.
\end{Rem}

\begin{ex}\label{hex1}
	Let $\vv w=(-3,1, 4, -2)\in W_4$. Then $\vv{tw}=(-3,-1, 4, -2)$, $\vv{wt}=(3,1, 4, -2)$, and 
	\[
	P(w)=P(wt)=\begin{tikzpicture}[scale=\domscale,baseline=-25pt]
	\hdom{0}{0}{1}
	\vdom{0}{1}{2}
	\vdom{1}{1}{3}
	\hdom{2}{0}{4}
	\end{tikzpicture}, \quad 
P(tw)=\begin{tikzpicture}[scale=\domscale,baseline=-25pt]
	\vdom{0}{0}{1}
	\vdom{0}{2}{2}
	\vdom{1}{0}{3}
	\hdom{2}{0}{4}
	\end{tikzpicture}.
	\]
They have the same hollow tableau
	\(\begin{tikzpicture}[scale=\domscale,baseline=-25pt]
\hobox{0}{0}{1}
\hobox{0}{2}{2}
\hobox{1}{1}{3}
\hobox{2}{0}{4}
	\end{tikzpicture}.\)
	
\end{ex}

Let $x\in\mathrm{Seq}_n (\mathbb{Z})$ be a sequence of integers with distinct absolute values. Define $tx$ to be the sequence obtained from $ x $ by changing the sign of $ \pm 1 $ in $ x $. In particular, if $\pm 1\not\in \{x_1,x_2,\cdots, x_n\}$, then $tx=x$. 


\begin{Lem}\label{lem1:dom1}
	Let $x\in\mathrm{Seq}_n (\mathbb{Z})$ be a sequence of nonzero integers with distinct absolute values. Assume that $ 1\in \{  x_1,x_2,\cdots, x_n\} $. Let $ a_k(x) $ $($resp. $ b_k(x) $$)$ be the sum of lengths of the first $k$ rows $($resp. columns$)$ of $ Q(x) $. Then for all $ k\geq 1 $,
	\[
	a_k(x)-a_k(tx)=0\text{ or }1, \quad b_k(x)-b_k(tx)=0\text{ or }-1.
	\]
\end{Lem}
\begin{proof}	
	
	By Proposition \ref{prop:shape}, $ p({}^-x) =\sh(Q(x))$, and hence 
	by Greene's theorem (Theorem \ref{gthm}), $a_k(x)$ can be explained as the maximal length of  subsequences which can be written as a union of $k$ increasing subsequences of ${}^-x$. In other words, there is a subsequence $\sigma$ of ${}^-x$, which is a disjoint union of increasing subsequences $\sigma_1, \cdots, \sigma_k$ with $a_k(x)=|\sigma_1|+\cdots+|\sigma_k|=|\sigma|$. 
	
	Set $y=tx$. First we claim that $a_k(y)\geq a_k(x)-1.$ In fact, if $(-1, 1)$ is not a subsequence of any $\sigma_i$, then $ t\sigma_i $ is also an increasing subsequence of ${}^-y$,  which implies that $a_k(y)\geq |t\sigma|=|\sigma|=a_k(x)$  by  Greene's theorem.
	Assume now that $(-1, 1)$ is a subsequence of  one of $\sigma_i$'s, say $ \sigma_1 $. Then $\sigma_1$ is not a subsequence of ${}^-y$. Deleting  $-1$ or $1$ from $\sigma_1$, we get an increasing subsequence $\sigma'_1$ of ${}^-y$. Then $\sigma'_1, \sigma_2, \cdots, \sigma_k$ form $k$ disjoint increasing subsequences of ${}^-y$, of total length $|\sigma|-1$. By   Greene's theorem, $a_k(y)\geq |\sigma|-1=a_k(x)-1$. 

	With $y=tx$, we obtain $-1\in \{  y_1,y_2,\cdots, y_n\} $. Thus $(-1, 1)$ is  not a subsequence of ${}^-y$. The previous argument (with $x$ replaced by $y$) shows that $a_k(x)\geq a_k(y)$. 
	
	Now the first statement is obtained. Similarly, one can prove $ b_k(y)-b_k(x)=0\text{ or }1 $, by a column version of Greene's theorem.
\end{proof}


\begin{Lem}\label{lem1:intersect1}
	Let $w\in W_n$. Then for any $1\leq j\leq n$,
	\[
	\dom_j(P(w))\cap\dom_j(P(wt))\neq\emptyset.
	\]
\end{Lem}
\begin{proof}
	In view of Theorem \ref{dthmis}, it is equivalent to $$ \dom_j(Q(z))\cap\dom_j(Q(tz))\neq\emptyset $$ for all $ z\in W_n $.

	Assume without loss of generality that $ 1\in \{ z(1), \cdots, z(n) \} $. Let 
	\[
	\mb{z}_k=(z(1), \cdots, z(k)).
	\]
	Set $\mb{y}_k=t\mb{z}_k$. If $ 1\in \{ z(1), \cdots, z(k) \} $, Lemma \ref{lem1:dom1} implies $a_l( \mb{z}_k)-a_l( \mb{y}_k)=0$ or $1$,  and $b_l( \mb{z}_k)-b_l( \mb{y}_k)=0$ or $-1$ for all $1\leq l\leq k $, otherwise $\mb{y}_k=\mb{z}_k$ and $a_l( \mb{z}_k)-a_l( \mb{y}_k)=b_l( \mb{z}_k)-b_l( \mb{y}_k)=0$.

By definition, $ Q( \mb{z}_{k}) $ (resp. $ Q( \mb{y}_{k}) $) is obtained from $ Q( \mb{z}_{k-1}) $ (resp. $Q( \mb{y}_{k-1})$) by adding $\dom_k(Q( \mb{z}_k))$ (resp. $\dom_k(Q( \mb{y}_k))$). It suffices to prove  that 
$$ \dom_k(Q( \mb{z}_k))\cap\dom_k(Q(\mb{y}_k))\neq\emptyset  $$ for all $ 1\leq k \leq n$.
	 
	 Assume that it does not hold for $ k $. Then there exists some $l>0$ such that one of  $\dom_k(Q(\mb{z}_k))$ and $\dom_k(Q(\mb{y}_k))$ is contained in the first $l$ rows (or columns) and  the other one is contained in other rows (or columns).
	 
When $\dom_k(Q(\mb{z}_k))$ is contained in the first $l$ rows, 
\[
a_l(\mb{z}_k)=a_l(\mb{z}_{k-1})+2,\quad a_l(\mb{y}_k)=a_l(\mb{y}_{k-1}), 
\]
which implies that 
	\[
	a_l(\mb{z}_k)-a_l(\mb{y}_{k})=2+a_l(\mb{z}_{k-1})-a_l(\mb{y}_{k-1})=2\ \mbox{or}\ 3,
	\]
	a contradiction with Lemma \ref{lem1:dom1}.

	When $\dom_k(Q(\mb{z}_k))$ is contained in the first $l$ columns,  
	 \[
	 b_l(\mb{z}_k)=b_l(\mb{z}_{k-1})+2,\quad b_l(\mb{y}_k)=b_l(\mb{y}_{k-1}),
	 \]
	  which implies that 
	\[
	b_l(\mb{z}_k)-b_l(\mb{y}_{k})=2+b_l(\mb{z}_{k-1})-b_l(\mb{y}_{k-1})=2\ \mbox{or}\ 1,
	\]
	also a contradiction with Lemma \ref{lem1:dom1}.
	
	 The remaining cases  deduce similar contradictions.
\end{proof}

In general, we also need to consider the domino tableau 
\begin{equation*}\label{dweq1}
	P(w)_k:=((\cdots((\emptyset\laa w(1))\laa w(2))\cdots)\laa w(k))
\end{equation*}
for $k\leq n$. Since the domino insertion depends only on the relative order of the corresponding numbers. Lemma \ref{lem1:intersect1} can be easily generalized as follows.

\begin{Lem}\label{lem1:intersect}
	Let $w\in W_n$. If $j$ occurs in $P(w)_k$, then
	\[
	\dom_j(P(w)_k)\cap\dom_j(P(wt)_k)\neq\emptyset.
	\]
\end{Lem}

Recall the domino insertion algorithm from Definition \ref{def:dom}.

\begin{Prop}\label{prop:holow-induction}
	Let $ D $, $ \overline{D} $ be two standard domino tableaux with $\caH(D)=\caH(\overline D)$.
	Let $E=D\laa ci$ and $\overline E=\overline D\laa  ci$ with $ c=\pm1 $. If $\dom_j(E)\cap\dom_j(\overline E)\neq\emptyset$ for all $ j\geq i $ occurred in $E$ and $\overline E$, then  $\caH(E)=\caH(\overline E)$.
\end{Prop}
\begin{proof}
	We use induction on $ j $ to prove that $ \mc{H}(E_{\leq j}) =\mc{H}(\overline E_{\leq j}) $ for  $j \geq i $. If $ j=i $, it is clear. Now assume that $ j>i $ and $ \mc{H}(E_{\leq j-1}) =\mc{H}(\overline E_{\leq j-1}) $. If $j$ does not occur in $D$ and $\overline D$, obviously $ \mc{H}(E_{\leq j}) =\mc{H}(\overline E_{\leq j}) $. If $ j $ occurs in $D$ and $\overline D$, $\caH(\dom_j(D))=\caH(\dom_j(\overline D))$. There are the following two cases.

	\textbf{Case} 1. $\mc{H}(\dom_j(D)) $  does not  overlap with $   \mc{H}(E_{\leq j-1})$. Then only cases (i) and (ii) in Definition \ref{def:dom} could happen, which implies that 
	\[
	\mc{H}(E_{\leq j})=   \mc{H}(E_{\leq j-1})\sqcup \mc{H}(\dom_j(D)).
	\]
	 Since $\mc{H}(\dom_j(\overline{D})) $  also does not  overlap with $   \mc{H}(\overline{E}_{\leq j-1})$,  
	\[
	\mc{H}(\overline{E}_{\leq j})=   \mc{H}(\overline{E}_{\leq j-1})\sqcup \mc{H}(\dom_j(\overline{D}).
	\]
	   Hence	$ \mc{H}(E_{\leq j})=\mc{H}(\overline{E}_{\leq j}) .$

\textbf{Case} 2. $\mc{H}(\dom_j(D)) $    overlaps with $   \mc{H}(E_{\leq j-1})$ at an even box $ (k,l) $. By (ii) and (iii) in Definition \ref{def:dom} , if $ \dom_j(D) $ is horizontal (resp. vertical), then $  \mc{H}(E_{\leq j})$ is obtained from $ \mc{H}(E_{\leq j-1}) $ by adding an even box filled with $ j $ to the next row (resp. column) of $ (k,l)$. It is similar for  $  \mc{H}(\overline E_{\leq j})$.

 If $  \dom_j(D) $ and $  \dom_j(\overline{D}) $ have the same direction, then $ \mc{H}(E_{\leq j})=\mc{H}(\overline{E}_{\leq j}) .$ 

Next assume without loss of generality that  $ \dom_j(D) $ is horizontal and $  \dom_j(\overline{D}) $ is vertical. Then $\mc{H}(\dom_j(E))=\mc{H}(E_{\leq j})\setminus \mc{H}(E_{\leq j-1}) $ is at an even box $ (k+1, r) $ with $ r\leq l+1$. Similarly,   $\mc{H}(\dom_j(\overline E))=\mc{H}(\overline{E}_{\leq j})\setminus \mc{H}(\overline{E}_{\leq j-1}) $ is at an even box $ (s,l+1)  $ with $ s\leq k+1$. 

 If $r=l+1 $,  then $(k+1, l+1)\in \mc{H}({E}_{\leq j})$,  and $ (a,l+1)\in  \mc{H}(E_{\leq j-1})=\mc{H}(\overline{E}_{\leq j-1}) $ for all $ 1\leq a\leq k $ with $a+l+1$ even by using Lemma \ref{lem1:leftup}. Since the even box $ (s,l+1) \not\in \mc{H}(\overline{E}_{\leq j-1})$, one has $s=k+1$. Similarly, if $ s=k+1 $, we get $ r=l+1 $. In these cases, we have $ \mc{H}(E_{\leq j})=\mc{H}(\overline{E}_{\leq j}) .$

Finally assume that $ s<k+1 $ and  $ r<l+1 $. This yields $s\leq k-1$ and $r\leq l-1$. We always have $\dom_j(E_{\leq j})\cap \dom_j(\overline{E}_{\leq j}) =\emptyset $, a contradiction.
\end{proof}

\noindent\textbf{The proof of Theorem \ref{thm:hollow}.}
	 Note that $(wt)(1)=-w(1)$ and $(wt)(k)=w(k)$ for $k>1$. With $P(w)_1=\emptyset\laa w(1)$ and $P(wt)_1=\emptyset\laa -w(1)$, obviously $\mc{H}(P(w)_1)=\mc{H}(P(wt)_1)$ is the even box $(1, 1)$ filled with $|w(1)|$. Applying Lemma \ref{lem1:intersect} and Proposition \ref{prop:holow-induction}, we can show that $\mc{H}(P(w)_k)=\mc{H}(P(wt)_k)$ by induction on $k$.  In particular, when $k=n$, one has $\mc{H}(P(w))=\mc{H}(P(wt))$. Moreover, $ \mc{H}(Q(w))=\mc{H}(Q(wt)) $ follows  immediately from the  definition of the $ Q $-tableau. Then we apply \eqref{eq:pqinv} to complete the proof.

%
%

\section{GK dimensions of highest weight modules}\label{sec:lie}

%
%

In this section, we will prove Theorem \ref{thm2} about GK dimensions of highest weight modules for classical types. 

We always assume that $\mu$ is an anti-dominant weight. Recall in \S2.4 that $J=\{\alpha\in\Delta_{[\mu]}\mid\langle\mu, \alpha^\vee\rangle=0\}$ and $W_{[\mu]}^J=\{w\in W_{[\mu]}\mid w\ \mbox{is shortest in}\ wW_J\}$. If $\mu$ is integral, then $J\subset\Delta$ and $W_{[\mu]}=W$.

\subsection{General notations} We assume that $\Phi=B_n, C_n$ or $D_n$ and apply the usual realization of $\Phi$ \cite{Hum78}. So $\Phi$ is a subset of a real vector space $E=\mathbb{R}^n$ which admits a natural $W$-invariant bilinear form $(-,-)$ with orthonormal basis $\ep_i$. Write the set $\Delta$ of simple roots as $\{\alpha_1,\cdots, \alpha_{n-1}, \alpha_n\}$ such that $\alpha_i=\ep_i-\ep_{i+1}$ for $1\leq i<n$ and $\alpha_n=\ep_n, 2\ep_n, \ep_{n-1}+\ep_n$, respectively. Let 
\begin{equation*}
\Lambda=\{\lambda\in\hs\mid\bil{\lambda}{\al}\in\mathbb{Z}\text{ for all }\al\in\Phi^+ \}
\end{equation*}
be the set of  \textit{integral} weights. Note that we can write $\lambda\in\hs$ as 
\begin{equation*}
\lambda=(\lambda_1, \cdots, \lambda_n)=\lambda_1\ep_1+\cdots\lambda_n\ep_n
\end{equation*}
with $\lambda_1, \cdots, \lambda_n\in\mathbb{C}$. 


Define action $ s_\al $ on $ \hs$ by $s_\al\lambda=\lambda-\bil{\lambda}{\al}\al\text{ for all } \lambda\in\hs$ and $\alpha\in\Phi$. Then the Weyl group $ W $ is a Coxeter group with  $ S=\{s_\al\mid \al\in\Delta \} $  being the set of generators. For $\Phi=B_n$ or $C_n$, we have an isomorphism between the Coxeter systems $ (W,S) $ and $ (W_n, S_n) $ (see \S\ref{subsec:bn}) via
\begin{equation*}
s_{\ep_i-\ep_{i+1}}\mapsto s_{n-i},0\leq  i<n,\text{ and } s_{2\ep_n}=s_{\ep_n}\mapsto t.
\end{equation*}
We identify these groups, and hence $s_{n-i}$ ($1\leq i<n $) acts on $ \mf{h}^* $ by interchanging the coefficients of $ \ep_i $ and $ \ep_{i+1} $, and $ t $ changing the sign of the coefficient of $ \ep_n $. More generally, if $ w\in W_n $, then 
\begin{equation}\label{key0}
	w(\ep_{n+1-k})=\sgn(w(k))\ep_{n+1-|w(k)|}, \text{ for all }1\leq k\leq n,
\end{equation}
where $ \sgn:\mathbb{R}\to \{0,1,-1\} $ is the usual sign function.
\begin{Lem}\label{lem:wrho}
	Let $ w\in W_n $ and $\delta=(-n, \cdots, -2, -1)$. Then
	\[
	w\delta=(-w^{-1}(n),-w^{-1}(n-1),,\cdots,-w^{-1}(1)).
	\]
\end{Lem}
\begin{proof}
	Note that $ w^{-1}(w(i))=i $. By \eqref{key0}, one has
	\begin{align*}
		w\delta&=-w(\sum_{i=1}^n i\ep_{n+1-i})=-\sum_{i=1}^n i\sgn(w(i))  \ep_{n+1-|w(i)|}\\
		&=-\sum_{i=1}^n w^{-1}(|w(i)|)\ep_{n+1-|w(i)|}=-\sum_{k=1}^n w^{-1}(k)\ep_{n+1-k},
	\end{align*} 
	which implies the lemma.
\end{proof}

\subsection{The integral case} In this subsection, the anti-dominant weight $\mu$ is always integral.

\begin{Lem}\label{spb}
	Suppose $\Phi=B_n$ or $C_n$. Let $\lambda=w\mu$ for $w\in W^J$. Then $p(\lambda^-)=p({}^-\vv{w})$.
\end{Lem}
\begin{proof}
	Choose  $\delta=(-n, \cdots, -1)$ as in Lemma \ref{lem:wrho}. Set $\nu=w\delta$. Then $\delta\in C_\circ$ and $\nu\in wC_\circ$. With $w\in W^J$, Lemma \ref{uclem2} yields $\lambda\in\widehat{wC_\circ}$. By Definition \ref{defuc}, this implies
	\begin{itemize}
		\item [(1)] if $ 1\leq i<j\leq n $ and $ \nu_i<\nu_j $, then $ \lambda_i\leq \lambda_j $;
		\item [(2)] if $ 1\leq i<j\leq n $ and $ \nu_i>\nu_j $, then $ \lambda_i>\lambda_j $;
		\item [(3)] if $ 1\leq i, j\leq n $ and $ \nu_i<-\nu_j $, then $ \lambda_i\leq -\lambda_j $;
		\item [(4)] if $ 1\leq i, j\leq n $ and $ \nu_i>-\nu_j$, then $ \lambda_i>-\lambda_j $.
	\end{itemize}
	Applying Lemma \ref{RSeq} to $\lambda^-$ and $\nu^-$, we get $p(\lambda^-)=p(\nu^-)$. Set $z=w^{-1}$. Lemma \ref{lem:wrho} implies $\nu^-={}^-\vv{z}$. Furthermore, one has $p({}^-\vv{z})=p({}^-\vv{w})$ in view of \eqref{eq:pqinv} and Proposition \ref{prop:shape}.  This proves the lemma.
\end{proof}

\begin{Lem}\label{spd}
	Suppose $\Phi=D_n$. Let $\lambda=w\mu$ for $w\in W^J$. Then $p(\lambda^-)$ is equal to $p({}^-\vv{w})$ or $p({}^-\vv{wt})$.
\end{Lem}
\begin{proof}
	Set $\delta=(-n, \cdots, -1)$. In this case, $W=W'_n$ and $tC_\circ=C_\circ$. In particular, both $\delta$ and $t\delta$ are contained in the anti-dominant chamber $C_\circ$. Note that the anti-dominant weight $\mu$ satisfies $\mu_1\leq\cdots\leq\mu_{n-1}\leq-|\mu_n|$ in this case. If $\mu_n\leq0$, then $\mu$ is also an anti-dominant weight of the system $B_n$. In view of Lemma \ref{uclem2} and Lemma \ref{spb}, we can find $z\in W_n$ such that $\lambda=z\mu$ is contained in the upper closure of the $B_n$-chamber which contains $z\delta$. Since $D_n$ can be treated as a subsystem of $B_n$, $\lambda$ is also contained in the upper closure of $D_n$-chamber which contains $z\delta$. This chamber is $zC_\circ$. On the other hand, Lemma \ref{uclem2} implies $\lambda\in\widehat{wC_\circ}$. Thus $wC_\circ=zC_\circ=ztC_\circ$. It is evident that exactly one of $\{z, zt\}$ belongs to $W'_n$. This forces $w=z$ or $zt$, in view of Lemma \ref{uclem1}. Lemma \ref{spb} then yields $p(\lambda^-)=p({}^-\vv{w})$ or $p({}^-\vv{wt})$. If $\mu_n>0$, the argument is similar with $\mu$ replaced by $t\mu$. 
\end{proof}

\begin{Prop}\label{prop:integral-algo}
	Let $\Phi=B_n$ or $C_n$ and $ \lambda\in\mathfrak{h}^*$ be integral. Then 
	\[
	\gkd L(\lambda)=n^2-F_b(\lambda^-).
	\]
\end{Prop}
\begin{proof}
	This is a consequence of Theorem \ref{thm1} (or Proposition \ref{prop:bB}), Proposition \ref{pr:main1}, Lemma \ref{spb} and the fact that $|\Phi^+|=n^2$.
\end{proof}

\begin{Prop}\label{prop:integral-algo-dn}
	Let $\Phi=D_n$ and $ \lambda\in\mathfrak{h}^*$ be integral. Then 
 \[
\gkd L(\lambda)=n^2-n-F_d(\lambda^-).
\]
\end{Prop}
\begin{proof}
	By Lemma \ref{spd}, one has $p(\lambda^-)=p({}^-\vv{w})$ or $p({}^-\vv{wt})$. Here we apply Theorem \ref{thm:hollow} and get $ p({}^-\vv{w})\ev=p({}^-\vv{wt})\ev=p(\lambda^-)\ev$. Then the lemma follows from Theorem \ref{thm1} (or Proposition \ref{prop:bd}), Proposition \ref{pr:main1} and the fact that $|\Phi^+|=n^2-n$. 
\end{proof}

\subsection{The general case}\label{non-int} In this section, we will consider the GK dimension of $ L(\lam) $ with $ \lam $ not necessarily integral. 

Now $\Phi=B_n$, $C_n$ or $D_n$. Fix $\lambda\in\mathfrak{h}^*\simeq\mathbb{C}^n$. First we describe a decomposition of the root system $ \Phi_{[\lambda]}=\{\alpha\in\Phi\mid\langle\lambda, \alpha^\vee\rangle\in\mathbb{Z}\}\subset\Phi $
into some orthogonal subsystems.

Let $K_{(z)}:=\{i\leq n\mid\lambda_i\in z+\mathbb{Z}\}$ for $z\in\mathbb{C}$.  
For each $ z\in \mathbb{C} $, we define a set $\Phi_{(z)}$  as follows.

In general, if $z\not\in\frac{1}{2}\mathbb{Z}$, we set
\[
\Phi_{(z)}:=\{\ep_i-\ep_j, \pm(\ep_j+\ep_k), \ep_k-\ep_l\mid i, j\in K_{(z)}, k, l\in K_{(-z)}, i\neq j, k\neq l\}.
\]
If $\Phi=B_n$ and $z\in\frac{1}{2}\mathbb{Z}$, we set
\begin{equation*}\label{inteq11}
	\Phi_{(z)}:=\{\pm(\ep_i\pm \ep_j), \pm \ep_k\mid i, j, k\in K_{(z)}\ \mbox{and}\ i<j\}.
\end{equation*}
If $\Phi=C_n$ and $z\in\mathbb{Z}$, we set
\begin{equation*}\label{inteq2}
	\Phi_{(z)}:=\{\pm(\ep_i\pm \ep_j), \pm 2\ep_k\mid i, j, k\in K_{(z)}\ \mbox{and}\ i<j\}.
\end{equation*}
If $\Phi=C_n$ and $z\in\frac{1}{2}+\mathbb{Z}$, we set
\begin{equation*}\label{inteq3}
	\Phi_{(z)}:=\{\pm(\ep_i\pm \ep_j)\mid i, j\in K_{(z)}\ \mbox{and}\ i<j\}.
\end{equation*}
If $\Phi=D_n$ and $z\in\frac{1}{2}\mathbb{Z}$, we set
\begin{equation*}\label{inteq3} 
	\Phi_{(z)}:=\{\pm(\ep_i\pm \ep_j)\mid i, j\in K_{(z)}\ \mbox{and}\ i<j\}.
\end{equation*}

For $z\in\mathbb{C}$, let $K_{(z)}=\{i_1<\cdots<i_r\}$ and $K_{(-z)}=\{j_1<\cdots<j_s\}$. Suppose that $ \Phi_{(z)} \neq \emptyset$.
If $z\not\in\frac{1}{2}\mathbb{Z}$, then $\Phi_{(z)}$ is generated by simple roots 
\begin{equation*}\label{ninteq1}
\{  \ep_{i_k}-\ep_{i_{k+1}} \mid 1\leq k<r  \}\cup \{ \ep_{i_r}+\ep_{j_{s}}\}\cup \{ \ep_{j_l}-\ep_{j_{l+1}} \mid 1\leq l<s     \},
\end{equation*}
so $\Phi_{(z)}$ is a root system of type $A_{r+s-1}$. If $\Phi=B_n$ (resp. $ D_n $) and $z\in\frac{1}{2}\mathbb{Z}$, $ \Phi_{(z)} $ is a root system of type $ B $ (resp. $ D $). 
If $\Phi=C_n$ and $z\in\mathbb{Z}$ (resp. $ z\in\frac{1}{2}+\mathbb{Z}$), $ \Phi_{(z)} $ is a root system of type $ C$ (resp. $ D $). Note that $ C_1=B_1=A_1 $, $ D_3=A_3 $, $ D_2=A_1\times A_1 $. Hence  $ \Phi_{(z)} $  is irreducible except when it is of type $ D_2 $.

As usual, $\Ree(z)$ is used to denote the real part of a complex number $z$.

\begin{Prop}\label{rclem5}
	Fix $\Phi=B_n, C_n$ or $D_n$. For $\lambda\in\mathfrak{h}^*$,
	\[
	\Phi_{[\lambda]}=\bigsqcup_{0\leq\Ree(z)\leq 1/2}\Phi_{(z)}.
	\]
Moreover, $ \Phi_{(z)} $, $ 0\leq\Ree(z)\leq 1/2 $ are orthogonal with each other.
%
\end{Prop}
\begin{proof}
For $\beta=\ep_i-\ep_j\in\Phi$ (resp. $\ep_i+\ep_j$, $\ep_i$, $2\ep_i\in\Phi$), $\beta\in\Phi_{[\lambda]}$ if and only if $\lambda_i-\lambda_j\in\mathbb{Z}$ (resp. $\lambda_i+\lambda_j$, $2\lambda_i$, $\lambda_i\in\mathbb{Z}$). This implies that  $\Phi_{(z)}\subset\Phi_{[\lambda]}$, and that any $\beta\in\Phi_{[\lambda]}$ must belong to some $  \Phi_{(z)}$. Hence $\Phi_{[\lambda]}$ is a union of all the $  \Phi_{(z)}$, $ z\in \mathbb{C} $.
	
For $z, z'\in\mathbb{C}$, if	$z+z'\in\mathbb{Z}$ or $z-z'\in\mathbb{Z}$ , then $\Phi_{(z)}=\Phi_{(z')}$. Thus for any $ z'\in\mathbb{C} $, there exists some $ z $ with $ 0\leq \Ree(z)\leq \frac12 $ such that $\Phi_{(z)}=\Phi_{(z')}$.

 If $z\pm z'\not\in\mathbb{Z}$, then $\Phi_{(z)}\cap\Phi_{(z')}=\emptyset$, and in fact they are orthogonal with each other. This completes the proof.	
\end{proof}

As above, $K_{(z)}=\{i_1<\cdots<i_r\}$ and $K_{(-z)}=\{j_1<\cdots<j_s\}$.
If $ z\notin \frac12\mathbb{Z} $,
set 
\begin{equation*}\label{ninteq2}
\lambda_{(z)}=(\lambda_{i_1}, \lambda_{i_2}, \cdots,\lambda_{i_r}, -\lambda_{j_s}, -\lambda_{j_{s-1}}, \cdots, -\lambda_{j_1})\in\mathrm{Seq}_{r+s}(z+\mathbb{Z}).
\end{equation*}
If $z\in\frac{1}{2}\mathbb{Z}$, we always set
\begin{equation*}\label{ninteq3}
	\lambda_{(z)}=(\lambda_{i_1}, \lambda_{i_2}, \cdots,\lambda_{i_r}),
\end{equation*}
and we have $\lambda_{(z)}^-\in\mathrm{Seq}_{2r}(z+\mathbb{Z})$.


\begin{Thm}\label{thm:algo}
Let $\lambda=(\lambda_1, \lambda_2, \cdots, \lambda_n)\in \mathfrak{h}^*$.
	\begin{itemize}
		\item [(1)] If $ \Phi=B_n $, then
		\[
		\gkd L(\lambda)=n^2-F_b(\lambda_{(0)}^-)-F_b(\lambda_{(\frac{1}{2})}^-)-\sum _{0\leq\Ree(z)\leq\frac{1}{2}, z\neq0,\frac{1}{2}} F_a(\lambda_{(z)}).
		\]
		\item [(2)] If $ \Phi=C_n $, then
		\[
		\gkd L(\lambda)=n^2-F_b(\lambda_{(0)}^-)-F_d(\lambda_{(\frac{1}{2})}^-)-\sum _{0\leq\Ree(z)\leq\frac{1}{2}, z\neq0,\frac{1}{2}} F_a(\lambda_{(z)}).
		\]
		\item [(3)] If $ \Phi=D_n $, then
		\[
		\gkd L(\lambda)=n^2-n-F_d(\lambda_{(0)}^-)-F_d(\lambda_{(\frac{1}{2})}^-)-\sum _{0\leq\Ree(z)\leq\frac{1}{2}, z\neq0,\frac{1}{2}} F_a(\lambda_{(z)}).
		\]
	\end{itemize}
\end{Thm}


\begin{proof}
Set $S=\{z\in\mathbb{C}\mid0\leq\Ree(z)\leq\frac{1}{2}\}$ and $\aff'=\aff_{[\lambda]}:W_{[\lam]}\to\mathbb{N}$.

 Let $W_{(z)}$ be the Weyl group corresponding to $\Phi_{(z)}$. Then $W_{[\lambda]}$ must be a direct product of subgroups $ W_{(z)}$ with $ z\in S $ by Proposition \ref{rclem5}.

Write $\lambda=w\mu$ for $\mu$ anti-dominant and $ w\in W_{[\mu]}^J$ (\S2.4). Let $ w=\prod_{z\in S} w_{(z)} $, where $ w_{(z)} $ is  the component in $ W_{(z)} $. Then $ \aff'(w)=\sum_{z\in S} \aff'(w_{(z)}) $. One can see that $ \aff'(w_{(z)})$ is given by $ F_a(\lambda_{(z)}) $, $ F_b(\lambda_{(z)}^-) $ or $ F_d(\lambda_{(z)}^-) $ according to the type of $ \Phi_{(z)} $ (using results in the integral cases). The theorem then follows from Proposition \ref{pr:main1}.

We should pay attention to some extreme cases. If $\lambda_{(z)}$ has exactly one entry for $\Phi=C_n$, $z=\frac{1}{2}$ or $\Phi=D_n$, $z=0, \frac{1}{2}$, the root system $\Phi_{(z)}$ is empty. Fortunately Example \ref{eq:F1} shows that $F_d(\lambda_{(z)}^-)=0$ in this case, which will not change our formulas.
\end{proof}

\begin{ex}
	Let $\Phi=C_{10}$. Assume that 
	\[
	\lambda=(3.1, 2.3,1.1, -4, -4.1,2.5,1.9,2,2.1,0)\in\mathfrak{h}^*.
	\]
	Set $x=\lambda_{(0.1)}=(3.1,1.1,2.1,-1.9, 4.1)$, $y=\lambda_{(0.3)}=(2.3)$, $z=\lambda_{(0)}=(-4,2,0)$,  $w=\lambda_{(0.5)}=(2.5)$. It can be verified that $p(x)=(3, 1^2)$, $p(y)=(1)$, $p(z^-)=(4, 1^2)$ and $p(w^-)=(1^2)$. Therefore $p(z^-)\od=(2, 1)$ and $p(w^-)\ev=(1)$. By Theorem \ref{thm:algo}, we have
	\[
	\begin{aligned}
		\gkd L(\lambda) &=n^2-F_a(x)-F_a(y)-F_b(z^-)-F_d(w^-)\\
		&=10^2-3-0-1-0=96.
	\end{aligned}	
	\]
\end{ex}


%
%

\section{Associated varieties for classical groups}\label{sec:HC}

%
%

In this section, we will determine the associated varieties of simple highest weight HC modules for the classical Lie groups. We refer to \S2.2 for necessary preliminaries about associated varieties and HC-modules. 

\subsection{The associated varieties}

Recall that $q(x)$ is the dual partition of $p(x)$ for $x\in\mathrm{Seq}_n (\Gamma)$ (see \S2.6). The case of type $A$ is already known.

\begin{Thm}[{\cite[Thm. 6.4]{BX}}]\label{thm:associatedA}
	Let $L(\lambda)$ be a HC module of $SU(k,n-k)$ $(1\leq k<n)$ with $\lambda=(\lambda_1,\cdots,\lambda_n)\in\mathfrak{h}^*$. Set $q=q(\lambda)=(q_1, \cdots, q_{n})$. Then  $V(L(\lambda))=\overline{\mathcal{O}}_{k(\lambda)}$ with
	\[
	k(\lambda)=\begin{cases}
		q_2 & \text{ if } \lambda\  \text{is integral},\\
		\min\{k, n-k\} & \text{ otherwise}.
	\end{cases}
	\]
\end{Thm}

Now we state our criteria for the other classical Hermitian groups.

\begin{Thm}\label{thm:associated}
	Let $L(\lambda)$ be a HC module of $G$ with $  \lambda=(\lambda_1,\cdots,\lambda_n)\in\mathfrak{h}^*$. Set $q=q(\lambda^-)=(q_1, \cdots, q_{2n})$. Then $V(L(\lambda))=\overline{\mathcal{O}}_{k(\lambda)}$ with $ k(\lambda) $ given as follows.

	\begin{itemize}
		\item [(1)] $ G=Sp(n,\mathbb{R}) $ with $ n\geq 2 $. Then 
		\[
		k(\lambda)=\begin{cases}
			2q_2\od & \text{ if } \lambda_1\in\mathbb{Z},\\
			2q_2\ev+1 & \text{ if } \lambda_1\in\frac12+\mathbb{Z},\\
			n & \text{ otherwise}.
		\end{cases}
		\]
		\item [(2)] $ G=SO^*(2n) $ with $ n\geq 4 $. Then 
		\[
		k(\lambda)=\begin{cases}
			q_2\ev & \text{ if } \lambda_1\in\frac12\mathbb{Z},\\
			\left\lfloor \frac{n}{2} \right \rfloor & \text{ otherwise}.
		\end{cases}
		\]
		\item [(3)] $ G=SO(2,2n-1) $ with $ n\geq 3 $.   Then\[
		k(\lambda)=\begin{cases}
			0& \text{ if } \lambda_1-\lambda_2\in\mathbb{Z}, \lambda_1>\lambda_2,\\
			1 & \text{ if } \lambda_1-\lambda_2\in\frac12+\mathbb{Z}, \lambda_1>0,\\
			2 & \text{ otherwise}.
		\end{cases}
		\]
		\item [(4)] $ G=SO(2,2n-2) $ with $ n\geq 4 $.   Then\[
		k(\lambda)=\begin{cases}
			0& \text{ if } \lambda_1-\lambda_2\in\mathbb{Z}, \lambda_1>\lambda_2,\\
			1 &  \text{ if } \lambda_1-\lambda_2\in \mathbb{Z}, -|\lambda_n|<  \lambda_1\leq \lambda_2\\
			2 & \text{ otherwise}.
		\end{cases}
		\]
			\end{itemize}
\end{Thm}

The idea of the proof is straightforward. First we compute the GK dimension of $L(\lambda)$ using Theorem \ref{thm:algo}, then compare it with the formula in Proposition \ref{C: dimYk}. The column algorithm in Lemma \ref{def:F} will be helpful.


\subsection{The GK dimensions}

Let $ \Phi_c $ be the set of compact roots, which is the root system of $ (\mf{k},\mf{h}) $.  Denote $  \Phi_c^+= \Phi_c\cap \Phi^+ $. Then it is generated by $\Delta_c:=\Delta\cap\Phi_c^+$. Recall that $ L(\lambda) $ is a HC module if and only if $\lambda$ is \textit{$\Phi_c^+$-dominant}, that is, $ \bil{\lambda}{\alpha} \in\mathbb{Z}_{>0}$ for all $ \alpha\in  \Phi_c^+$. We always put $q=q(\lambda^-)=(q_1, q_2\cdots, q_{2n})$ when $\Phi_{[\lambda]}$ is a root system of type $B$, $C$ or $D$.

\begin{Lem}\label{prop:1}
Let $ G=\Sp $ $(n\geq 2) $ and $ L(\lambda) $ be a HC module of $G$. 
\begin{itemize}
	\item [(1)] If $ \lambda_1\in\mathbb{Z} $, then $ \gkd L(\lambda)=q_2\od(2n+1-2q_2\od)$.
	\item [(2)] If $ \lambda_1\in \frac12+\mathbb{Z} $, then $ \gkd L(\lambda)=n+q_2\ev(2n-1-2q_2\ev)$.
	\item [(3)] If $ \lambda_1\notin \frac12 \mathbb{Z} $, then $ \gkd L(\lambda)=\frac{n(n+1)}{2}$.
\end{itemize}
\end{Lem}
\begin{proof}
In this case, $ \Phi_c^+=\{ \ep_i-\ep_j\mid i<j \}$ (e.g., \cite{EHW}). We get $\lambda_1>\cdots>\lambda_n$ and $\lambda_i-\lambda_j\in\mathbb{Z}$ since $\lambda$ is $\Phi_c^+$-dominant. 

If $\lambda_1\in \mathbb{Z} $, Theorem \ref{thm:algo} yields $\gkd L(\lambda)=n^2-F_b(\lambda^-)$. By Greene's Theorem, $ q(\lambda^-)=(q_1, q_2)$. It follows from Lemma \ref{def:F} that
\[
\gkd L(\lambda)=n^2-(q_1\od)^2-q_2\od(q_2\od-1)=q_2\od(2n+1-2q_2\od),
\]
where the last equality follows from the fact that $q_1\od+q_2\od=n$.

If $\lambda_1\in \frac12+\mathbb{Z} $, then $\gkd L(\lambda)=n^2-F_d(\lambda^-)$. Applying Lemma \ref{def:F}, a similar argument shows that
\[
\gkd L(\lambda)=n^2-q_1\ev(q_1\ev-1)-(q_2\ev)^2=n+q_2\ev(2n-1-2q_2\ev).
\]

If $\lambda_1\not\in \frac12 \mathbb{Z} $, then $\gkd L(\lambda)=n^2-F_a(\lambda)$. By Greene's Theorem and Lemma \ref{def:F}, we have 
\[
\gkd L(\lambda)=n^2-F_a(\lambda)=n^2-\frac{n(n-1)}{2} =\frac{n(n+1)}{2}.
\]
\end{proof}

\begin{Lem}\label{prop:2}
	Let $ G=SO^*(2n)$ $(n\geq 4) $ and $ L(\lambda) $ be a HC module of $G$. 
	\begin{itemize}
		\item [(1)] If $ \lambda_1\in \frac12\mathbb{Z} $, then $ \gkd L(\lambda)=q_2\ev(2n-1-2q_2\ev)$.
		\item [(2)] If $ \lambda_1\notin \frac12 \mathbb{Z} $, then $ \gkd L(\lambda)=\frac{n(n-1)}{2}$.
	\end{itemize}
\end{Lem}
\begin{proof} In this case,  $ \Phi_c^+ =\{ \ep_i-\ep_j\mid i<j \}$. We get $\lambda_1>\cdots>\lambda_n$ and $\lambda_i-\lambda_j\in\mathbb{Z}$.

If $ \lambda_1\in \frac12\mathbb{Z} $, Theorem \ref{thm:algo} implies $ \gkd L(\lambda)=n^2-n-F_d(\lambda^-)$, which is equal to  $ q_2\ev(2n-1-2q_2\ev) $ in view of Lemma \ref{def:F}.

If $ \lambda_1\notin \frac12 \mathbb{Z} $, then $\gkd L(\lambda)=n^2-n-F_a(\lambda)=\frac{n(n-1)}{2}$ using Theorem \ref{thm:algo} and Lemma \ref{def:F}.
\end{proof}

\begin{Lem}\label{prop:3}
	Let $ G=SO(2,2n-1)$ $(n\geq 3) $ and $ L(\lambda) $ be its HC module. 
	\begin{itemize}
		\item [(1)] If $ \lambda_1-\lambda_2\in \mathbb{Z} $, then $ \gkd L(\lambda)=\begin{cases}
		0&\text{ if } \lambda_1>\lambda_2,\\
		2n-1 & \text{ if } \lambda_1\leq \lambda_2.
		\end{cases}$
		\item [(2)] If $ \lambda_1-\lambda_2\in \frac12 +\mathbb{Z} $, then $ \gkd L(\lambda)=\begin{cases}
		2n-2&\text{ if } \lambda_1>0,\\
		2n-1 & \text{ if } \lambda_1\leq 0.
		\end{cases}$
		\item [(3)] If $ \lambda_1-\lambda_2\notin \frac12 \mathbb{Z} $, then $ \gkd L(\lambda)=2n-1$.
	\end{itemize}
\end{Lem}
\begin{proof} 
	Now $ \Phi_c^+=\{ \ep_i\pm \ep_j\mid   2\leq  i<j\leq n \}\cup \{\ep_i \mid 2\leq  i\leq n \} $. Then $ \lambda_2>\lambda_3>\cdots>\lambda_n>0$, $ \lambda_2\in \frac12\mathbb{Z} $ and $ \lambda_i-\lambda_j\in \mathbb{Z} $ for all $ i,j\geq 2 $. Set $x=(\lambda_1)$ and $y=(\lambda_2, \cdots, \lambda_n)$.
	
	If $ \lambda_1-\lambda_2\in \mathbb{Z} $, Theorem \ref{thm:algo} shows that $\gkd L(\lambda)=n^2-F_b(\lambda^-)$. By Greene's Theorem, if $ \lambda_1>\lambda_2 $, $ q(\lambda^-)=(2n) $; if $ \lambda_1\leq \lambda_2 $, $ q(\lambda^-)=(2n-2,2) $ or $ (2n-2, 1^2)  $. Then (1) follows from Lemma \ref{def:F}.
	
	If $ \lambda_1-\lambda_2\in \frac12+\mathbb{Z} $, then $\gkd L(\lambda)=n^2-F_b(x^-)-F_b(y^-)$. It is easy to verify that $ q(x^-)=(1^2)$ (when $\lambda_1\leq 0$) or $(2)$ (when $\lambda_1>0$). Moreover, $q(y^-)=(2n-2)$. Lemma \ref{def:F} yields (2). 
	
	If $ \lambda_1-\lambda_2\notin \frac12\mathbb{Z} $, then  $\gkd L(\lambda)=n^2-F_a(x)-F_b(y^-)$. We obtain $q(x)=(1)$ and $q(y^-)=(2n-2)$. Then (3) follows from Lemma \ref{def:F}.
\end{proof}

\begin{Lem}\label{prop:4}
	Let $ G=SO(2,2n-2)$ $(n\geq 4) $ and $ L(\lambda) $ be its HC module. 
	\begin{itemize}
		\item [(1)] If $ \lambda_1-\lambda_2\in \mathbb{Z} $, then $ \gkd L(\lambda)=\begin{cases}
		0&\text{ if } \lambda_1>\lambda_2,\\
		2n-3 & \text{ if } -|\lambda_n|<  \lambda_1\leq \lambda_2,\\
		2n-2 &\text{ if }   \lambda_1\leq -|\lambda_n|.
		\end{cases}$
		\item [(2)] If $ \lambda_1-\lambda_2\notin \mathbb{Z} $, then $ \gkd L(\lambda)=2n-2$.
	\end{itemize}
\end{Lem}
\begin{proof} 
	Now $ \Phi_c^+=\{ \ep_i\pm\ep_j\mid 2\leq i<j\leq n \}$. We obtain
	$ \lambda_2>\lambda_3>\cdots>|\lambda_n|$, $ \lambda_2\in \frac12\mathbb{Z} $ and $ \lambda_i-\lambda_j\in \mathbb{Z} $ for all $ i,j\geq 2 $. Set $x=(\lambda_1)$ and $y=(\lambda_2, \cdots, \lambda_n)$. 

	First assume that $ \lambda_1-\lambda_2\in \mathbb{Z} $. Theorem \ref{thm:algo} yields $ \gkd L(\lambda)=n^2-n-F_d(\lambda^-)  $. If $ \lambda_1>\lambda_2 $, then $ q(\lambda^-)=(2n)$ (when $\lambda_n>0$) or $(2n-1, 1)$ (when $\lambda_n\leq 0$). In both cases, $F_d(\lambda^-)=n^2-n$. If $ -|\lambda_n|< \lambda_1\leq \lambda_2 $, then $ q(\lambda^-)=(2n-2, 2)$ or $(2n-3,3)$, we get $F_d(\lambda^-)=n^2-3n+3$. If $ \lambda_1\leq -|\lambda_n| $, then $ q(\lambda^-)=(2n-2, 1^2)$ or $(2n-3, 1^3)$, one has $F_d(\lambda^-)=n^2-3n+2$. Then (1) follows.

	Next assume that $ \lambda_1-\lambda_2\in \frac12+\mathbb{Z} $. Theorem \ref{thm:algo} gives  $ \gkd L(\lambda)=n^2-n-F_d(x^-)  -F_d(y^-) $. With $q(x^-)=(2)$ or $(1^2)$, we always have $F_d(x^-)=0$. By Greene's Theorem, $q(y^-)=(2n-2)$ or $(2n-3,1)$. In either case $F_d(y^-)=(n-1)(n-2)$.

	Finally assume that $ \lambda_1-\lambda_2\notin \frac12\mathbb{Z} $. Then $ \gkd L(\lambda)=n^2-n-F_a(x)  -F_d(y^-)$. The argument is similar.
\end{proof}

\noindent\textbf{Proof of Theorem \ref{thm:associated}.} By Proposition \ref{C: dimYk} and Table \ref{tab1}, we get the following facts.
	
	For $ G=Sp(n,\mathbb{R}) $,  
	$ 		\dim \overline{\mathcal{O}}_k=\frac12k(2n-k+1) \text{ with } 0\leq k\leq n. $
	
	For $ G=SO^*(2n) $, 
	$\dim \overline{\mathcal{O}}_k=k(2n-2k-1)\text{ with } 0\leq k\leq \left\lfloor\frac{n}{2}\right\rfloor .$
	
	For $ G=SO(2,2n-1) $, $
	\dim \overline{\mathcal{O}}_k=0,2n-2,2n-1 \text{ with } k=0,1,2.$
	
	For $ G=SO(2,2n-2) $, $
	\dim \overline{\mathcal{O}}_k=0,2n-3,2n-2 \text{ with } k=0,1,2.$
	
	Note that $\dim \overline{\mathcal{O}}_k $ is monotone in $ k $.	
	Now the  theorem follows immediately  from Lemmas \ref{prop:1}-\ref{prop:4} and the fact that $ \gkd L(\lambda)=\dim \overline{\mathcal{O}}_{k(\lambda)}$.
	
%
%

\section{Associated varieties for Exceptional groups} \label{sec:exp}

%
%

In this section, we give the associated varieties of highest weight HC modules for the two exceptional groups. Now $\Phi$ is of type $E_6$ or $E_7$. As usual $\Phi$ can be realized as a subset of $\mathbb{R}^8$. The simple roots are 
\begin{align*}
	\alpha_1&=\frac{1}{2}(\ep_1-\ep_2-\ep_3-\ep_4-\ep_5-\ep_6-\ep_7+\ep_8) ,\\
	\alpha_2&=\ep_1+\ep_2,\\
	\alpha_k&=\ep_{k-1}-\ep_{k-2} \text{ with }3\leq k\leq n,
\end{align*}
where $n=6, 7$. 
Here we adopt notation in \cite{EHW}. 
In particular, if $\Phi=E_6$, then $\Delta_c=\Delta\backslash\{\alpha_1\}$; if $\Phi=E_7$, then $\Delta_c=\Delta\backslash\{\alpha_7\}$. 

\begin{figure}[htpb]
	\begin{tikzpicture}[scale=1.5,baseline=0]
		\filldraw (0,0.05) node[above=1pt]{$ 1 $} circle [radius=0.05];
		\draw (0.05,0.05)--++(0.65,0); 
		\draw (0.75,0.05) node[above=1pt]{$ 3 $} circle[radius=0.05];  
		\draw (0.8,0.05)--++(0.65,0);
		\draw  (1.45,0) ++(0.05,0.05) node[above=1pt]{$ 4$} circle[radius=0.05];
		\draw (1.55,0.05)--++(0.65,0);
		\draw  (2.2,0) ++(0.05,0.05) node[above=1pt]{$ 5$} circle[radius=0.05];
		\draw (2.3,0.05)--++(0.65,0);
		\draw  (2.95,0) ++(0.05,0.05) node[above=1pt]{$ 6$} circle[radius=0.05];
		\draw (1.5,0)--++(0,-0.5);
		\draw  (1.45,-0.6) ++(0.05,0.05) node[right=1pt]{$ 2$} circle[radius=0.05];
	\end{tikzpicture}
	\quad
	\begin{tikzpicture}[scale=1.5,baseline=0]
		\draw (0,0.05) node[above=1pt]{$ 1 $} circle [radius=0.05];
		\draw (0.05,0.05)--++(0.65,0); 
		\draw (0.75,0.05) node[above=1pt]{$ 3 $} circle[radius=0.05];  
		\draw (0.8,0.05)--++(0.65,0);
		\draw  (1.45,0) ++(0.05,0.05) node[above=1pt]{$ 4$} circle[radius=0.05];
		\draw (1.55,0.05)--++(0.65,0);
		\draw  (2.2,0) ++(0.05,0.05) node[above=1pt]{$ 5$} circle[radius=0.05];
		\draw (2.3,0.05)--++(0.65,0);
		\draw  (2.95,0) ++(0.05,0.05) node[above=1pt]{$ 6$} circle[radius=0.05];
		\draw (3.05,0.05)--++(0.65,0);
		\filldraw  (3.7,0) ++(0.05,0.05) node[above=1pt]{$ 7$} circle[radius=0.05];
		\draw (1.5,0)--++(0,-0.5);
		\draw  (1.45,-0.6) ++(0.05,0.05) node[right=1pt]{$ 2$} circle[radius=0.05];
	\end{tikzpicture}
	\caption{Diagrams of $ E_6 $ and $ E_7 $ }
	\label{d67}
\end{figure}
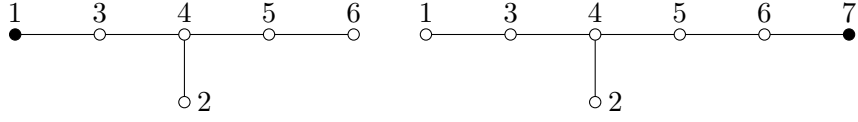

\subsection{The associated varieties} 

Recall from \eqref{eq:phi-lam} that, for an integral weight $ \lam\in\hs $, we have $ \Psi_\lam^+=\{\alpha\in\Phi^+\mid\langle\lambda, \alpha^\vee\rangle>0\} $.
\begin{Thm}\label{67}
	Let $L(\lambda)$ be a HC module of $G$ of type $E_6/E_7$. If $\lambda\in\hs$ is not integral, then $V(L(\lambda))=\overline{\mathcal{O}}_{k(\lambda)}$ with $k(\lambda)=2$ $($for $E_6$$)$ or $3$ $($for $E_7$$)$. In the case when $\lambda$ is integral, $ k(\lam) $ is given as follows.
	\begin{itemize}
		\item [(1)] If $G$ is of type $E_6$, then
		\[
		k(\lambda)=\begin{cases}
			0& \text{ if } \Psi_{\lam}^+\cap S_2\neq\emptyset,\\
			1& \text{ if } \Psi_{\lam}^+\cap S_1\neq\emptyset \text{ and } \Psi_{\lam}^+\cap S_2=\emptyset ,\\
			2& \text{ if } \Psi_{\lam}^+\cap S_1=\emptyset,
		\end{cases}
		\]
		where $S_1=\{\frac{1}{2}(\pm(\ep_1+\ep_2+\ep_3-\ep_4)-\ep_5-\ep_6-\ep_7+\ep_8)\}$ and $S_2=\{\alpha_1\}$.
	
	\item [(2)] If $G$ is of type $E_7$, then
		\[
	k(\lambda)=\begin{cases}
		0& \text{ if } \Psi_{\lam}^+\cap S_3\neq\emptyset,\\
		1& \text{ if } \Psi_{\lam}^+\cap S_2\neq\emptyset \text{ and } \Psi_{\lam}^+\cap S_3=\emptyset ,\\
		2& \text{ if } \Psi_{\lam}^+\cap S_1\neq\emptyset \text{ and } \Psi_{\lam}^+\cap S_2=\emptyset ,\\
		3& \text{ if } \Psi_{\lam}^+\cap S_1=\emptyset,
	\end{cases}
	\]
	where $S_1=\{\frac{1}{2}(\pm(\ep_1-\ep_2-\ep_3+\ep_4)-\ep_5+\ep_6-\ep_7+\ep_8), \ep_5+\ep_6\}$, $S_2=\{\pm\ep_1+\ep_6\}$ and $S_3=\{-\ep_5+\ep_6\}$.
	\end{itemize}

\end{Thm}

\begin{ex}
	Let $ \Phi=E_6 $ and $ \lam=\rho-2\alpha_1=(-1, 2, 3, 4, 5, -3, -3, 3)$. Since $ \lam $ is  $ \Phi_c^+ $-dominant, $L(\lambda)$ is a HC module. It can be easily verified that  $S_1\subset\Psi_{\lam}^+$ and $\alpha_1\not\in\Psi_\lambda^+$.	Therefore $V(L(\lambda))=\overline{\mathcal{O}}_{1}$ by Theorem \ref{67}(1).
\end{ex}

\begin{Rem}
	When $L(\lambda)$ is a HC module of exceptional types, we know that its associated variety are determined by the corresponding $\aff$-functions, in view of Propositions \ref{C: dimYk} and \ref{pr:main1}. The computation of $\aff$-functions for exceptional Weyl groups is by no means an easy problem even using computers. However, as shown in the above example, with Theorem \ref{67}, the associated variety can be easily obtained by hand. 
\end{Rem}

The rest of this section is devoted to the proof of Theorem \ref{67}.

\subsection{Some facts} 
Let $W_c$ be the Weyl group of the compact root system $\Phi_c$. Denote by $W^c$ the set of minimal length coset representatives of $W_c\backslash W$. Then $W=W_cW^c$. Let $ w\in W$. For simplicity, we write (see \S2.5)
\[
\Phi_w^+:=\Phi_{wC_\circ}^+=\{\alpha\in \Phi^+\mid w^{-1}\alpha<0 \}.
\]
Let $ w=s_{\gamma_1}s_{\gamma_2} \cdots s_{\gamma_k}$ be a reduced expression of $ w $ with $\gamma_1, \cdots, \gamma_k\in\Delta$. For $ 1\leq r\leq k $, define positive roots
\[
\beta_r=s_{\gamma_1}s_{\gamma_2} \cdots s_{\gamma_{r-1}} \gamma_r.
\]
Then $ \beta_r $'s are distinct and (see for example \cite[\S1.7]{Hum90}) 
\begin{equation}\label{gfeq1}
\Phi_w^+=\{\beta_{r}\mid   1\leq r\leq k \}.
\end{equation}
In particular, this set is independent of the choice of the reduced expressions of $ w $. Recall the Bruhat ordering ``$\leq$'' on $W$ (e.g., \cite[\S5.9]{Hum90}).

\begin{Lem}\label{hermp}
	Let $G$ be a Lie group of Hermitian type. Suppose that $y\leq w$ for $y, w\in W^c$. Write $z=y^{-1}w$. Then $ \ell(w)=\ell(y)+\ell(z) $, and  \begin{equation}\label{eq:Phi-w}
	\Phi_{w_cw}^+=\Phi_{w_cy}^+\sqcup  w_cy(\Phi_z^+) .
	\end{equation}
\end{Lem}
\begin{proof}
	If $ \ell(w)=\ell(y)+1$, one can write $w=ys_\beta$ for some $\beta\in\Phi^+$. It is well-known that (see for example \cite{BC86}) in all Hermitian cases we may always choose $\beta$ simple. Thus $z=s_\beta$ and $ \ell(w)=\ell(y)+\ell(z) $. If $ \ell(w)>\ell(y)+1$, Deodhar \cite{Deo77} showed that there is a sequence $y=y_0<y_1<\cdots<y_k=w$ in $W^c$ such that $\ell(y_i)=\ell(y_{i-1})+1$ for $1\leq i\leq k$. The previous argument implies that $y_{i-1}^{-1}y_i$ are simple reflections. This gives the first result. The second statement is an easy consequence of the first one and \eqref{gfeq1}. 
\end{proof}

The following result is due to \cite[Lem. 3.17]{EHW}.

\begin{Lem}\label{exlem1}
	Suppose that $G$ is of type $E_6$ or $E_7$. Let $L(\lambda)$ be a highest weight HC module of $G$. If $\lambda$ is not integral, then $\gkd L(\lambda)=|\Phi^+\backslash\Phi_c^+|$.
\end{Lem}

With this lemma,  it suffices to consider the case when $ \lam $ is integral.

\subsection{A decomposition of $W^c$}

Denote by $w_\circ$ (resp. $w_c$) the longest element in $W$ (resp. $W_c$). All the $\Phi_c^+$-dominant weights with infinitesimal character $-\rho$ can be written as $w_cw(-\rho)$ with $w\in W^c$. According to  \cite[Lem. 7.1]{CIS}, the set $ W^c $ can be partitioned into some subsets such that $ \gkd L(-w_cw\rho)$ is constant when $ w$ varies in each subset. 
 
Indeed, we label elements of $ W^c $ in the same way as \cite[Figs. 7.1, 7.2]{CIS}, where  
$ W^c=\{w_0,w_1,\cdots, w_{26}\} $ for type $ E_6 $ and $ W^c=\{w_0,w_1,\cdots, w_{55}\} $
for type $ E_7 $.
Reduced expressions of these elements can be easily read from  \cite[Figs. 7.1, 7.2]{CIS}, keeping in mind that our labels on simple roots (see Figure \ref{d67}) are different from those
in \cite[page 81]{CIS}. For example, in type $ E_6 $, a reduced expression of $ w_9 $  is given by $ s_{\alpha_1}s_{\alpha_3}s_{\alpha_4}s_{\alpha_5}s_{\alpha_2}s_{\alpha_6} $.


Denote by $ W^c_{\leq w}  $ the subset $ \{ y\in W^c\mid y\leq w \}$ for any $ w\in W^c$. In \cite[Fig. 7.1]{CIS}, the set $ W^c $  for type $ E_6 $ is decomposed into three subsets $  \msc{C}_0=W^c_{\leq w_{19}}$, $  \msc{C}_1=W^c_{\leq w_{25}}\backslash W^c_{\leq w_{19}} $ and $  \msc{C}_2=W^c_{\leq w_{26}}\backslash W^c_{\leq w_{25}}=\{w_{26}\}$. Similarly, in \cite[Fig. 7.2]{CIS}, the set $ W^c $  for type $ E_7 $ is decomposed into 4 disjoint subsets $ \msc{C}_0=W^c_{\leq w_{23}}$, $\msc{C}_1=W^c_{\leq w_{48}}\backslash W^c_{\leq w_{23}}$, $\msc{C}_2=W^c_{\leq w_{54}}\backslash   W^c_{\leq w_{48}} $ and $\msc{C}_3=W^c_{\leq w_{55}}\backslash   W^c_{\leq w_{54}}=\{w_{55}\}$.

\begin{Lem}\label{lemlast}
	Recall the sets $S_k$ defined in Theorem \ref{67}.
	\begin{itemize}
		\item [(1)] Let $\Phi=E_6$. Fix $0\leq k\leq 2$. The following conditions are equivalent:
		\begin{itemize}
			\item [(1a)] $z\in\msc{C}_k$;
			\item [(1b)] $\gkd L(-w_cz\rho)=(2-k)(8+3k)=|\Phi^+|-\aff(w_cz)$;
			\item [(1c)] $\Phi_{w_cz}^+\cap S_k\neq\emptyset$ when $k\geq1$ and $\Phi_{w_cz}^+\cap S_{k+1}=\emptyset$ when $k\leq1$.
		\end{itemize}
		\item [(2)] Let $\Phi=E_7$. Fix $0\leq k\leq 3$. The following conditions are equivalent:
		\begin{itemize}
			\item [(2a)] $z\in\msc{C}_k$;
			\item [(2b)] $\gkd L(-w_cz\rho)=(3-k)(9+4k)=|\Phi^+|-\aff(w_cz)$;
			\item [(2c)] $\Phi_{w_cz}^+\cap S_k\neq\emptyset$ when $k\geq1$ and $\Phi_{w_cz}^+\cap S_{k+1}=\emptyset$ when $k\leq2$.
		\end{itemize}
	\end{itemize}
\end{Lem}

\begin{proof}
	Here we only prove the case $k=0$ for $\Phi=E_7$, while the proofs for the other cases are similar. In this case, $W^c\backslash\msc{C}_0$ has three minimal elements $ w_{16} ,w_{24}, w_{26} $ under the Bruhat order. Moreover, \cite[Figs. 7.2]{CIS} shows that
	\begin{equation}\label{eq:w16}
	w_{16}=w_{13}s_{\alpha_7}\geq w_{13}, w_{24}=w_{21}s_{\alpha_2}\geq w_{21}\text{ and }w_{26}=w_{23}s_{\alpha_5}\geq w_{23},
\end{equation}
	with $w_{13}, w_{21}, w_{23}\in \msc{C}_0$. In view of \eqref{eq:Phi-w}, we denote by $\beta_{16}$, $\beta_{24}$ and  $\beta_{26}$ the unique roots in $\Phi_{w_cw_{16}} ^+\setminus \Phi_{w_cw_{13}} ^+$, $ \Phi_{w_cw_{24}} ^+\setminus \Phi_{w_cw_{21}} ^+$ and $ \Phi_{w_cw_{26}} ^+\setminus \Phi_{w_cw_{23}} ^+$, respectively. By \eqref{eq:Phi-w} and \eqref{eq:w16}, we can get
	\begin{equation*}\label{lemleq1}
	\beta_{16}=w_cw_{13}\alpha_7, \beta_{24}=w_cw_{21}\alpha_2\ \mbox{and}\ \beta_{26}=w_cw_{23}\alpha_5
	\end{equation*}
	A straightforward computation shows that $\{\beta_{16}, \beta_{24}, \beta_{26}\}=S_1$. It can be also directly verified that the set $\Phi_{w_cw_{23}}^+\cap S_1=\emptyset$.
	
	The equivalence of (2a) and (2b) is due to \cite[Lemma 7.1]{CIS} and Proposition \ref{pr:main1}. If (2a) holds, that is, $z\in\msc{C}_0$, then $z\leq w_{23}$ since $w_{23}$ is the unique maximal element in $\msc{C}_0$. Lemma \ref{hermp} yields $\Phi_{w_cz}^+\subset\Phi_{w_cw_{23}}^+$ and thus $\Phi_{w_cz} ^+\cap S_1=\emptyset$. This gives (2c). Conversely, suppose $\Phi_{w_cz} ^+\cap S_1=\emptyset$. Recall that $ w_{16} ,w_{24}, w_{26} $ are all the minimal elements in $W^c\backslash\msc{C}_0$. If $z\not\in \msc{C}_0$, then at least one of $\Phi_{w_cw_{16}} ^+$, $\Phi_{w_cw_{24}} ^+$ and $\Phi_{w_cw_{26}} ^+$ is contained in $\Phi_{w_cz} ^+$, in view of   \eqref{eq:Phi-w}. Thus $\Phi_{w_cz} ^+\cap S_1\neq\emptyset$, a contradiction.
\end{proof}

\noindent\textbf{Proof of Theorem \ref{67}.} If $\lambda$ is not integral, we can apply Proposition \ref{C: dimYk} and Lemma \ref{exlem1}. 

Now assume that $\lambda$ is integral, then there exists a unique $w\in W$ such that $\lambda\in\widehat{wC_\circ}\subset\overline{wC_\circ}$. Thus $\mu=w^{-1}\lambda\in\overline C_\circ$ is anti-dominant. Proposition \ref{pr:main1} and Lemma \ref{uclem2} imply $\gkd L(\lambda)=|\Phi^+|-\aff(w)$ and $\Psi_\lambda^+=\Phi_{wC_\circ}^+=\Phi_{w}^+$. Note that $L(\lambda)$ is a HC module if and only if $\lambda$ is $\Phi_c^+$-dominant, that is, $\Phi_c^+\subset\Psi_\lambda^+=\Phi_{w}^+$. This means $w=w_cz$ for some $z\in W^c$. The theorem then follows from Proposition \ref{C: dimYk} and Lemma \ref{lemlast}.

\bibliography{GK}

\end{document}